\documentclass[a4paper]{amsart}
   \usepackage{amssymb} 

\usepackage[utf8]{inputenc}
\usepackage[T1]{fontenc} 

\usepackage{microtype} 
\usepackage[hidelinks,pagebackref]{hyperref} 
   \hypersetup{final} 

\usepackage{mathtools} 

\usepackage[shortlabels]{enumitem}

\newtheorem{theorem}{Theorem}[section]
\newtheorem{lemma}[theorem]{Lemma}
\newtheorem{corollary}[theorem]{Corollary} 
\newtheorem{proposition}[theorem]{Proposition} 

\theoremstyle{definition}

\theoremstyle{remark}
\newtheorem{remark}[theorem]{Remark}

\numberwithin{equation}{section}

\newcommand*{\RR}{\mathbb{R}}
\newcommand*{\NN}{\mathbb{N}}

\DeclareMathOperator{\EE}{\mathbb{E}} 

\DeclareMathOperator{\Ent}{Ent}	
\DeclareMathOperator{\Var}{Var}	

\newcommand*{\sgn}{\operatorname{sgn}}

\title{Functional inequalities for two-level concentration}

\author[F. Barthe]{Franck Barthe}
\address{Institut de Math\'ematiques de Toulouse; UMR 5219 \\ Universit\'e de Toulouse; CNRS
\\UPS - F-31062 Toulouse Cedex 9, France}
\email{barthe@math.univ-toulouse.fr}

\author[M. Strzelecki]{Micha{\l} Strzelecki}
\address{Institute of Mathematics, University of Warsaw, Banacha 2, 02--097 Warsaw, Poland.}
\email{michalst@mimuw.edu.pl}

\thanks{Research partially supported by the National Science Centre, Poland, grants
no.\ 2015/19/N/ST1/00891 (MS) and 2017/24/T/ST1/00323 (doctoral scholarship of MS)}

\date{October 1, 2019} %



\begin{document}

\begin{abstract}
Probability measures satisfying a Poincar\'e inequality are known to enjoy a dimension-free 
concentration inequality with exponential rate. 
A celebrated result of Bobkov and Ledoux shows that  a Poincar\'e
inequality automatically implies a modified logarithmic Sobolev inequality. As a consequence the Poincar\'e inequality ensures a stronger dimension-free concentration property, known as two-level concentration. We show that a similar phenomenon occurs for the Lata{\l}a--Oleszkiewicz inequalities, which were devised to uncover dimension-free concentration with rate between exponential and Gaussian. Motivated by the search for counterexamples to related questions, we also develop analytic techniques to study functional inequalities for probability measures on the line with wild potentials. 
\end{abstract}

\maketitle

\section{Introduction}

This article is a contribution to the functional approach to concentration inequalities, see, e.g., \cite{MR1849347}. We work in the setting of Euclidean spaces $\big(\RR^d, \langle\cdot,\cdot\rangle,|\cdot|\big)$, although most of the results extend to more general settings as Riemannian manifolds. 
 
First we recall the main functional inequalities which allow to establish concentration properties. A probability measure $\mu$ on $\RR^d$ satisfies a logarithmic Sobolev inequality if there is a constant $C_{LS}<\infty$ such that for all smooth functions $f\colon \RR^d\to \RR$, 
\begin{equation}
\label{eq:LSI-intro}
\Ent_\mu(f^2) \le C_{LS} \int_{\RR^d} |\nabla f|^2 d\mu,
\end{equation}
where $\Ent_\mu(g)=\int g\log g \, d\mu-\left(\int g\, d\mu\right) \log \left(\int g\, d\mu\right)$ is the entropy of a nonnegative function. It is convenient to denote by $C_{LS}(\mu)$ the smallest possible constant $C_{LS}$ in \eqref{eq:LSI-intro}. Classically the inequality tensorizes, meaning $C_{LS}(\mu^{\otimes n})=C_{LS}(\mu)$ for all $n$, and the standard Gaussian measure satisfies a logarithmic Sobolev inequality. Conversely, measures with a log-Sobolev inequality enjoy a dimension-free concentration inequality with Gaussian rate: for all $n\ge 1$ and for all measurable $A\subset \RR^{nd}$ with $\mu^{\otimes n}(A)\ge \frac12$, it holds for all $t>0$
\[
\mu^{\otimes n}\big( A+t B_2\big)\ge 1-e^{-K\frac{t^2}{C_{LS}}},
\]
where $K$ is a numerical constant and $B_2=B_2^{nd}$ denotes the Euclidean unit ball (of $\RR^{nd}$). This concentration property can also be formulated in terms of deviations of functions, as we will mention later. 

The other main property in the field is the Poincar\'e inequality. A probability measure $\mu$ on $\RR^d$ enjoys a Poincar\'e inequality if there exists a constant $C_P<\infty$ such that for all smooth $f\colon\RR^d\to \RR$,  
\begin{equation}
\label{eq:Poinc-intro}
\Var_\mu(f)\le C_P \int_{\RR^d} |\nabla f|^2 d\mu
\end{equation}
where $\Var_\mu(f)=\int f^2 d\mu-\left(\int f\, d\mu\right)^2$ is the variance of $f$ with respect to $\mu$. Again $C_P(\mu)$ denotes the minimal constant for which the inequality holds. The Poincar\'e inequality tensorizes ($C_P(\mu^{\otimes n})=C_P(\mu)$) and ensures dimension-free concentration properties with exponential rate: for all $n$ and all $A\subset \RR^{nd}$ with $\mu^{\otimes n}(A)\ge \frac 12$, and for all $t>0$,
\begin{equation}
\label{eq:exp-concentration}
\mu^{\otimes n}\big( A+t B_2\big)\ge 1-e^{-K\frac{t}{\sqrt{C_P}}},
\end{equation}
where $K$ is a numerical constant. The symmetric exponential distribution $d\nu(t)=e^{-|t|}dt/2$
on $\RR$ satisfies a Poincar\'e inequality, but not the log-Sobolev inequality. In \cite{MR1122615}, Talagrand proved a stronger concentration property than the above one: if $\nu^{\otimes n} (A)\ge \frac12$ then for all $t>0$,
\begin{equation}
\label{eq:Talagrand}
 \nu^{\otimes n}\big( A+\sqrt{t} B_2+t B_1\big)\ge 1-e^{-t/C},
\end{equation}
for some universal constant $C$, where $B_p=B_p^n=\{x\in \RR^n;\; \sum_{i=1}^n |x_i|^p\le 1\}$.
Talagrand's two-level concentration inequality  \eqref{eq:Talagrand} is doubly sharp: for $A=(-\infty,0]\times \RR^{n-1}$ it captures the exponential behaviour of coordinates marginals of $\nu^{\otimes n}$, while for $A=\{x;\; \sum_i x_i\le 0\}$  using that $B_1^n\subset\{x\in \RR^n; \; \sum_i x_i\le 1\}$ and $B_2^n\subset\{x\in \RR^n; \; \sum_i x_i\le \sqrt n\}$, one gets 
\[ \nu^{\otimes n}\left( \left\{ x\in \RR^n;\; \frac{\sum_{i=1}^n x_i}{\sqrt n} \le \sqrt{t}+\frac{t}{\sqrt{n}} \right\} \right) \ge 1-e^{-t/C},\]
which is asymptotically of the right order in $t$ when $n\to \infty$, according to the Central Limit Theorem. 

Talagrand's two-level concentration phenomenon was incorporated in the functional approach by Bobkov and Ledoux: they introduced a modified log-Sobolev inequalities which implies concentration of the type \eqref{eq:Talagrand}. More importantly, they showed that it is implied
by the Poincar\'e inequality (which means that the concentration consequences of that inequality are stronger than \eqref{eq:exp-concentration}). More precisely, 
\begin{theorem}[Bobkov--Ledoux \cite{MR1440138}]
\label{theo:Boblov-Ledoux}
Let $\mu$ be a probability measure on $\RR^d$, which satisfies a Poincar\'e inequality with constant $C_P$. Then for any $c\in (0, C_P^{-1/2})$, there exists $K(c,C_P)<\infty$ such that for all smooth $f\colon \RR^d\to (0,\infty)$ such that pointwise $\left| \frac{\nabla f}{f}\right| \le c$,
\[
\Ent_\mu(f^2) \le K(c,C_P) \int_{\RR^d} |\nabla f|^2 d\mu.
\]
\end{theorem}

The Central Limit Theorem and an argument of Talagrand \cite{MR1122615} roughly imply
that if dimension-free concentration in Euclidean spaces occurs, then the rate of concentration 
cannot be faster than Gaussian, and the measure should be exponentially integrable. In this sense, Poincar\'e and log-Sobolev inequalities describe the extreme dimension-free properties.
The functional approach to  concentration with intermediate rate (between exponential and Gaussian) was developed by Lata{\l}a and Oleszkiewicz \cite{MR1796718}.
We say that a probability measure $\mu$ on $\RR^d$ satisfies the \emph{Lata{\l}a--Oleszkiewicz inequality}  with parameter $r\in [1,2]$, if there exists a constant $C_{LO(r)}<\infty$ such that for every smooth $f\colon \RR^d\to\RR$  one has
\begin{equation}
 \label{eq:LO}
 \sup_{\theta\in(1,2)} 
 \frac{\int_{\RR^d} f^2 d\mu - \Bigl(\int_{\RR^d} |f|^\theta d\mu\Bigr)^{2/\theta}}{(2-\theta)^{2(1-1/r)}}
 \leq C_{LO(r)} \int_{\RR^d} |\nabla f|^2 d\mu.
\end{equation}
Let us stress here that most of the information is encoded in the speed at which $(2-\theta)^{2(1-1/r)}$ vanishes as $\theta\to 2^-$ (by omitting the supremum on the left-hand side of~\eqref{eq:LO} and only considering a fixed $\theta\in(1,2)$ one gets a~significantly weaker inequality).
We sometimes omit  the dependence in $r$ in the notation when there is no ambiguity on the value of $r$. For $r=1$ the inequality is equivalent to Poincar\'e inequality.
 For $r=2$ and the Gaussian measure, such  inequalities  were first considered by Beckner  \cite{MR954373}. Moreover, up to the constants, the inequality for $r=2$ is equivalent to the log-Sobolev inequality (note in particular that the limit as $\theta\to 2^-$ of the ratio on the left-hand side is the entropy). Lata{\l}a and Oleszkiewicz proved that the above functional inequality tensorizes and implies dimension-free concentration with rate $\exp(-t^r)$: under
 \eqref{eq:LO}, if $\mu^{\otimes n}(A)\ge \frac12$ then for $t>0$,
 \[
 \mu^{\otimes n}(A+tB_2)\ge 1-e^{-K \left(t/\sqrt{C_{LO(r)}}\right)^r}.
 \]
 (their proof yields $K=1/3$; see also \cite{MR2127729} and Section 6 of \cite{MR2320410} for an extension to a more general setting).
By $\mu_r$ we denote the probability measure on the real line with density
\begin{equation*}
 d\mu_r(t) = \frac{e^{-|t|^r} dt}{2\Gamma(1+1/r)}, \quad t\in\RR.
\end{equation*}
For $r\in(1,2)$,  Lata{\l}a and Oleszkiewicz \cite{MR1796718} showed that  $\mu_r$ satisfies the inequality~\eqref{eq:LO} with a uniformly bounded constant (in this case $d=1$). For these measures, one obtains a dimension-free concentration inequality with a rate corresponding to the tails.

Another approach was suggested by Gentil, Guillin, and Miclo  \cite{MR2198019}, which we present now.
For $r\in (1, 2]$, we say that a probability measure $\mu$ on $\RR^d$ satisfies the \emph{modified log-Sobolev inequality} with parameter $r$ if there exists a constant $C_{mLS(r)}<\infty$ such that for every smooth function $f\colon \RR^d\to(0,\infty)$ one has
\begin{equation}
\label{eq:mLS}
 \Ent_{\mu}(f^2)
 \leq C_{mLS(r)} \int_{\RR^d} H_{r'}\Bigl( \frac{|\nabla f|}{f} \Bigr) f^2 d\mu,
 \end{equation}
where $H_{r'}(t) \coloneqq \max\{t^2, |t|^{r'}\}$ for $t\in\RR$ and $r'\ge 2$ is the dual exponent 
of $r$, defined by $\frac1r+\frac{1}{r'}=1$. This is a natural extension 
of the modified log-Sobolev inequality of Bobkov and Ledoux, see Theorem \ref{theo:Boblov-Ledoux}, which appears as the limit case $r=1$. Indeed, when $r\to 1^+$, $r'\to \infty$ and 
\begin{equation*}
  \lim_{r'\to\infty} H_{r'}(t/c) = 
 \begin{cases}
  t^2/c^2 & \text{if } |t|\leq c,\\
  \infty & \text{if } |t| > c.
 \end{cases}
\end{equation*}

The modified log-Sobolev inequality tensorizes as follows: if $\mu$ satisfies~\eqref{eq:mLS}, then for any positive integer $n$ and every smooth function $f\colon \RR^{dn}\to(0,\infty)$ one has
\begin{equation*}
 \Ent_{\mu^{\otimes n}}(f^2)
 \leq C_{mLS(r)} \int_{\RR^{dn}} \sum_{i=1}^n H_{r'}\Bigl( \frac{|\nabla_i f|}{f} \Bigr) f^2 d\mu^{\otimes n },
 \end{equation*}
where $\nabla_i f$ denotes the partial gradient with respect to the $i$-th $d$-tuple of coordinates of $\RR^{dn}$. It is proved in \cite{MR2198019} that for each $r\in(1,2)$ the measure $\mu_r$ satisfies 
a modified log-Sobolev inequality \eqref{eq:mLS} with parameter $r$.
This allows to recover a two-level concentration inequality of Talagrand \cite{TalConcMeasure},
extending \eqref{eq:Talagrand}: for $r\in (1,2)$, if $\mu_r^{\otimes n}(A)\ge \frac12$, then
\[  \mu_r^{\otimes n}\big( A+\sqrt{t} B_2+t^{\frac1r} B_r\big)\ge 1-e^{-t/C_r}. \]


\medskip

In view of Theorem \ref{theo:Boblov-Ledoux}, it is natural to conjecture, that similarly the Lata{\l}a--Oleszkiewicz inequality implies the modified log-Sobolev inequality (and therefore improved two-level concentration), cf. Remark 21 in \cite{MR2430612}.

\section{Main result and organization of the article}

We are ready to state our main result. Let us emphasize that it is not restricted to measures on the real line and that the dimension $d$ does not enter into the dependence of constants.

\begin{theorem}
\label{thm:main:Toulouse}
Let $r\in(1,2)$.
Let $\mu$ be
a probability measure on $\RR^d$ which satisfies the Lata\l{}a--Oleszkiewicz inequality~\eqref{eq:LO} with parameter $r$ and with constant $C_{LO(r)}$. Then $\mu$ satisfies the modified log-Sobolev inequality~\eqref{eq:mLS} with parameter $r$ and a constant  $C$ depending only on $C_{LO(r)}$ and~$r$. 
More precisely, one can take $C=\alpha(r) \max\{ C_{LO(r)},C_{LO(r)}^{1/r}\}$ for some function $\alpha$. 
\end{theorem}

 The concentration property can also be formulated in terms of functions. As shown in \cite{MR1796718}, the Lata{\l}a--Oleszkiewicz inequality \eqref{eq:LO} implies that  for any  integer $n\ge 1$ and every $1$-Lipschitz function $f\colon\RR^{dn}\to\RR$, one has
\begin{equation}
\label{eq:conc-LO}
 \mu^{\otimes n}\Bigl( \bigl|f-\int_{\RR^{dn}} f\, d\mu^{\otimes n} \bigr| \geq t\sqrt{C_{LO(r)}} \Bigr)
 \leq 2 \exp(-K \min\{t^2, t^r\}).
\end{equation}
The modified log-Sobolev inequality \eqref{eq:mLS} implies via a modification of Herbst's argument a stronger deviation inequality.
Therefore our main theorem ensures that the Lata{\l}a--Oleszkiewicz inequality ensures such
an improved concentration. This is the content of the next two corollaries, for which some notation is needed.
%
%
For $x=(x_1,\ldots,x_n)\in(\RR^d)^n$ and $p\in(1,\infty)$ denote
\begin{equation*}
 \|x \|_{p,2} \coloneqq \Bigl(\sum_{i=1}^n |x_i|^p\Bigr)^{1/p}
\end{equation*}
(here $|\cdot|$ stands for the $\ell_2$ norm on $\RR^d$; in the notation we suppress the roles of $d$ and $n$, but they will always be clear from the context).


\begin{corollary}
\label{cor:Toul-conc-functions}
Let $r\in (1,2)$. 
Let $\mu$ be
a probability measure on $\RR^d$ which satisfies the Lata\l{}a--Oleszkiewicz inequality~\eqref{eq:LO} with parameter $r$ and with constant $C_{LO}$. Then there exists a constant $C>0$, depending only on $C_{LO}$ and $r$, such that for any positive integer $n$, any smooth $f\colon \RR^{dn}\to\RR$, and all $t>0$,
\begin{equation*}
\begin{split}
&\mu^{\otimes n}\Bigl( \bigl|f-\int_{\RR^{dn}} f d\mu^{\otimes n} \bigr| \geq t \Bigr)\\
&\quad  \leq 2 \exp\Bigl(-\frac12 \min\Bigl\{
\frac{t^2}{C\sup_{x\in\RR^{dn}} |\nabla f(x)|^2}, \frac{t^r}{C^{r-1}\sup_{x\in(\RR^{d})^n} \|\nabla f(x) \|_{r',2}^r}\Bigr\}\Bigr).
\end{split}
\end{equation*}
\end{corollary}

Using standard smoothing arguments one can also obtain a result for not necessarily smooth functions, expressed in terms of their Lipschitz constants.

\begin{corollary}
\label{cor:Toul-conc-functions-Lip}
Let $r\in(1,2)$. Let $\mu$ be
a probability measure on $\RR^d$ which satisfies the Lata\l{}a--Oleszkiewicz inequality~\eqref{eq:LO} with parameter $r$ and constant $C_{LO}$. Then there exists a constant $C>0$, depending only on $C_{LO}$ and $r$, such that for any positive integer  $n$ the following holds: if $f\colon \RR^{dn}\to\RR$ satisfies
\begin{align*}
 |f(x) - f(y)| &\leq L_2 |x-y|,\\
 |f(x) - f(y)| &\leq L_{r,2} \|x-y\|_{r,2},
\end{align*}
for all $x,y\in (\RR^d)^n$, then for all $t>0$
\begin{equation*}
\mu^{\otimes n}\Bigl( \bigl|f-\int_{\RR^{dn}} f d\mu^{\otimes n} \bigr| \geq t \Bigr)
\leq 2 \exp\Bigl(-\frac12 \min\Bigl\{
\frac{t^2}{CL_2^2}, \frac{t^r}{C^{r-1}L_{r,2}^r}\Bigr\}\Bigr).
\end{equation*}
\end{corollary}

One can also express concentration in terms of enlargements of sets. Below $B_2^{dn}$ and $B_r^{dn}$ stand for the unit balls in the $\ell_2$ and $\ell_r$-norms on $\RR^{dn}$, respectively. Also, let
\begin{equation*}
B_{r,2}^{n,d} \coloneqq \Bigl\{(x_1,\ldots, x_n) \in(\RR^d)^n : \bigl(\sum_{i=1}^n |x_i|^r\bigr)^{1/r} \leq 1 \Bigr\}
\end{equation*}
be the unit ball in the norm $\|\cdot\|_{r,2}$.
Observe that for $r\in(1,2)$,
 \begin{equation*}
 d^{1/2-1/r}  B_{r,2}^{n,d}\subset B_r^{dn} \subset B_{r,2}^{n,d} \subset B_2^{dn} \subset n^{1/r-1/2}  B_{r,2}^{n,d}.
 \end{equation*}

\begin{corollary}
\label{cor:Toul-conc-set}
Let $\mu$ be
a probability measure on $\RR^d$ which satisfies the Lata\l{}a--Oleszkiewicz inequality~\eqref{eq:LO} with parameter $r$ and  constant $C_{LO}$. Then there exists a constant $K>0$, depending only on $C_{LO}$ and $r$, such that for any positive integer $n$ and any set $A\subset \RR^{dn}$ with $\mu^{\otimes n}(A)\geq 1/2$,
\begin{equation*}
\mu^{\otimes n}\Bigl( A +\Bigl\{(x_1,\ldots,x_n)\in(\RR^d)^n : \sum_{i=1}^n \min\{|x_i|^2,|x_i|^r\} \leq t \Bigr\} \Bigr)
 \geq 1-e^{- Kt}.
\end{equation*}
In particular,
\begin{equation}
\label{eq:cor-conc-set-mixed-ball}
\mu^{\otimes n}\bigl( A +\sqrt{t} B_2^{dn} + t^{1/r} B_{r,2}^{n,d} \bigr)
 \geq 1-e^{- Kt}.
\end{equation}
One can take $K=\frac{1}{32}\min\{1/C,1/C^{r-1}\}$, where $C$ is taken from Theorem \ref{thm:main:Toulouse}. 
\end{corollary}

This corollary should be compared with the results of Gozlan \cite{MR2682264}.
He proved that if a probability  measure $\mu$ on $\RR^d$ satisfies the Lata\l{}a--Oleszkiewicz inequality, then it satisfies a Poincar\'e type  inequality involving  a non-standard length the gradient (see Corollary~5.17 in \cite{MR2682264}), which in turn implies a slightly different type of two-level concentration (see Proposition~2.4 and Proposition~1.2 in \cite{MR2682264}). However, unlike in the above two corollaries, the constants which appear in his formulations do depend on the dimension $d$ of the underlying space (even though they do not depend on $n$).
Namely, if we denote $x_i = (x_i^1, \ldots, x_i^d)\in\RR^d$ for $i=1,\ldots, n$, then \cite{MR2682264}  shows the existence of  a constant $K>0$ (depending only on $C_{LO}$ and $r$) such that for any positive integer $n$ and any set $A\subset \RR^{dn}$ with $\mu^{\otimes n}(A)\geq 1/2$,
\begin{align*}
\mu^{\otimes n}\Bigl( A +\Bigl\{(x_1,\ldots,x_n)\in(\RR^d)^n : \sum_{i=1}^n \sum_{j=1}^d \min\Bigl\{\Bigl|\frac{x_i^j}{d}\Bigr|^2,\Bigl|\frac{x_i^j}{d}\Bigr|^r\Bigr\} \leq t \Bigr\} \Bigr)\\
\geq 1-e^{-Kt/d}
\end{align*}
(the $d$ in the denominator on the left-hand side comes from Corollary 5.17 of \cite{MR2682264} and the $d$ on the right-hand side---from Proposition~2.4 therein). In particular, this implies
\begin{equation*}
\mu^{\otimes n}\bigl( A +d^{3/2}\sqrt{t} B_2^{dn} + d^{1+1/r} t^{1/r} B_{r}^{dn} \bigr)
 \geq 1-e^{-Kt}.
\end{equation*}
In terms of the dependence on $d$ this is weaker than \eqref{eq:cor-conc-set-mixed-ball}, since
\begin{equation*}
B_{r,2}^{n,d} \subset d^{1/r-1/2} B_{r}^{dn} \subset d^{1+1/r} B_{r}^{dn},
\end{equation*}
with strict  inclusions when $d\geq 2$.

 \medskip
 
The organization of the rest of the article is the following. In Section~\ref{sec:LO-prelim} we introduce two more functional inequalities. They will serve as intermediate steps between the Lata{\l}a--Oleszkiewicz and the modified log-Sobolev inequalities. In Section~\ref{sec:proof-of-main-and-cor} we prove the main result and Corollaries~\ref{cor:Toul-conc-functions} and~\ref{cor:Toul-conc-set}. 
The rest of the paper deals with measures on the real line. One motivation is to make progress on a question that we do not fully settle: our main theorem shows an implication between two properties; are they actually equivalent? In Section~\ref{sec:R}, we recall the known criteria.
In Section~\ref{sec:weighted} we consider the weighted log-Sobolev inequalities
used by \cite{MR2198019} in order to derive \eqref{eq:mLS}. We show that the two properties are not equivalent. Workable criteria are available for measures on $\RR$ with strictly increasing potential close to $\infty$. In the final section, we develop an elementary approach to deal with potential with vanishing derivatives or even decreasing parts. We illustrate this method on the functional inequalities of interest.

\section{Preliminaries: a few more inequalities}
\label{sec:LO-prelim}

We start with the following observation.

\begin{lemma}
\label{lem:LO-to-Poincare}
 Suppose that  a probability measure $\mu$  on $\RR^d$  satisfies the Lata\l{}a--Oleszkiewicz inequality~\eqref{eq:LO} with constant $C_{LO}$. Then it satisfies the Poincar\'e inequality \eqref{eq:Poinc-intro} with constant $C_P = C_{LO}$.
\end{lemma}

\begin{proof}
By taking $\theta \to 1^+$ in~\eqref{eq:LO} we see that \eqref{eq:Poinc-intro}  holds for all positive smooth functions (with constant $C_{LO}$). Since the variance is translation invariant, we conclude that \eqref{eq:Poinc-intro}  holds for all smooth functions bounded from below. The general case follows by approximation.
\end{proof}

\begin{remark}
 Alternatively, one can deduce the Poincar\'e inequality from the fact that Inequality~\eqref{eq:LO} implies dimension-free concentration and the results of \cite{MR3311918}.
\end{remark}

  For $r\in(1,2)$ denote
 \begin{equation*}
  F_r(t) = \log^{2/r'}(1+t)-\log^{2/r'}(2), \quad t\geq 0.
 \end{equation*} 
We say that a probability measure $\mu$ on $\RR^d$ satisfies an \emph{$F_r$-Sobolev inequality} if there exists $C$ such that for every smooth $g\colon\RR^d\to\RR$,
 \begin{equation}
 \label{eq:FSob}
  \int_{\RR^d} g^2 F_r\Bigl(\frac{g^2}{\int_{\RR^d} g^2 d\mu} \Bigr) d\mu 
  \leq C \int_{\RR^d} |\nabla g|^2 d\mu.
 \end{equation}
 This inequality is \emph{tight}, i.e., we have equality for constant functions (if $f$ is constant and equal to zero on its support, then the expression $0/0$ should be interpreted as $0$ here and in  \eqref{eq:Itau} below).
 We say that $\mu$ on $\RR^d$ satisfies a \emph{defective $F_r$-Sobolev inequality} if there exists $B$ and $C$ such that for every smooth $g\colon\RR^d\to\RR$,
 \begin{equation}
 \label{eq:defective_FSob}
  \int_{\RR^d} g^2 F_r\Bigl(\frac{g^2}{\int_{\RR^d} g^2 d\mu} \Bigr) d\mu 
  \leq B \int_{\RR^d} g^2 d\mu +  C \int_{\RR^d} |\nabla g|^2 d\mu.
 \end{equation}
 
In \cite{MR2320410} Barthe, Cattiaux, and Roberto provided capacity criteria for, among others, the Lata\l{}a--Oleszkiewicz and $F_r$-Sobolev inequalities. We refer to Section~5 of \cite{MR2320410} for a thorough overview of the topic, and in particular to the  diagram on page 1041, which we use as a road map. The following theorem is a direct corollary of the results contained therein (and also in Wang's independent paper \cite{MR2127729}).

\begin{theorem}
  \label{thm:LO-equiv-forms}
  Let $r\in(1,2)$,  and let $\mu$ be an absolutely continuous probability measure $\mu$  on $\RR^d$. 
 Assume that  $\mu$ satisfies the Lata\l{}a--Oleszkiewicz inequality~\eqref{eq:LO} with parameter $r\in(1,2)$ and constant $C_{LO}$.
 Then $\mu$ satisfies the (tight) $F_r$-Sobolev inequality \eqref{eq:FSob} with a constant $C\le 1152 \,C_{LO}$.
 \end{theorem}

\begin{proof}  Denote $T(s)\coloneqq s^{2(1-1/r)}$. Recall the following definition of capacity: for Borel sets $A\subset \Omega\subset \RR^d$,  define
\[
 \operatorname{Cap}_\mu (A,\Omega)
 \coloneqq \inf \Big\{ \int_{\RR^d} |\nabla f|^2 d\mu : f_{|A} \geq 1 \text{ and } f_{|\Omega^c} =0 \Big\}
\]
(the infimum is taken over all locally Lipschitz functions), and 
\[
 \operatorname{Cap}_\mu (A)
 \coloneqq \inf \big\{ \operatorname{Cap}_\mu (A,\Omega) : A\subset \Omega \text{ and } \mu(\Omega)\leq 1/2\big\}
. \]
Theorem 18 and Lemma 19 of \cite{MR2320410} imply that if $\mu$ satisfies the Lata\l{}a--Oleszkiewicz inequality \eqref{eq:LO} with some constant $C_{LO}$, then
\[
 \mu(A) \frac{1}{T(\frac{1}{\log(1+1/\mu(A))})}  = \mu(A) \log^{2/r'}\big(1+1/\mu(A)\big)
 \leq  6 \, C_{LO} \operatorname{Cap}_\mu(A)
\]
for every $A\subset \RR^d$ with $\mu(A) < 1/2$. 
Using $2\log(1+t)\ge \log(1+2t)$, $t\ge 0$ and $r\in(1,2)$, we obtain that for all $A$ 
as above,   
\[ \mu(A) \log^{2/r'}\big(1+2/\mu(A)\big) \le 12 \, C_{LO} \operatorname{Cap}_\mu(A) .\]
Theorem 28 of  \cite{MR2320410} then gives\footnote{The assumption of absolute continuity of $\mu$ comes into play at this point. Indeed, the proof of Theorem 28 in  \cite{MR2320410} relies on a decomposition of $\RR^d$ into level sets $\{ f^2>\rho_k\}$, for some well chosen $\rho_k$ (cf. proof of Theorem 20 in \cite{MR2320410}), and one needs to know that the sets $\{f^2 = \rho_k\}\cap \{ |\nabla f| \neq 0 \}$ are negligible.}  that for every smooth function $f\colon \RR^d\to\RR$, 
 \[
  \int f^2 \log^{2/r'}(1+f^2)  d\mu - \Big(\int f^2  d\mu\Big)  \log^{2/r'}\Big(1+\int f^2 d\mu \Big) 
  \leq 1152 C_{LO}  \int |\nabla f|^2 d\mu.
 \]
  Substituting $f^2 = g^2/\int_{\RR^d} g^2 d\mu$ yields the claimed $F_r$-Sobolev inequality \eqref{eq:FSob}. 
\end{proof}

\begin{remark}
\label{rem:tutti-frutti}
The latter theorem remains valid even if $\mu$  is not  absolutely continuous. To see this we use an approximation argument.
Let $\gamma_\varepsilon$ be the centered Gaussian measure on $\RR^d$ with covariance matrix $\varepsilon \operatorname{Id}$. For small enough $\varepsilon>0$, $\gamma_\varepsilon$ satisfies the Lata\l{}a--Oleszkiewicz inequality with the same constant as $\mu$ and hence, by tensorization, so does $\mu\otimes\gamma_\varepsilon$.
Testing the inequality with  the function $(x,y)\mapsto f(x+y)$, we conclude that $\mu\ast\gamma_\varepsilon$ also satisfies the Lata\l{}a--Oleszkiewicz inequality (with a constant at most $2 C_{LO}(\mu)$ when $\varepsilon$ is small enough). Thus, by Theorem~\ref{thm:LO-equiv-forms}, $\mu\ast\gamma_\varepsilon$ satisfies the $F_r$-Sobolev inequality.
We fix a bounded smooth Lipschitz function, take $\varepsilon\to 0$, and arrive at the conclusion that $\mu$ satisfies the $F_r$-Sobolev inequality for all bounded smooth Lipschitz functions (we have pointwise convergence and since the function is Lipschitz and bounded we can use the dominated convergence theorem). Now if $f$ is an arbitrary smooth function such that $\int_{\RR^d} |\nabla f|^2 d\mu <\infty$, then it suffices to consider functions $f_n=\Psi_n(f)$, where $\Psi_n\colon\RR\to\RR$ is, say, an odd and non-decreasing function defined by
\begin{equation*}
 \Psi_n(t) =
  \begin{cases}
\-\Psi_n(-t) & \text{ for } t<0,\\
 t & \text{ for } t\in[0,n),\\
 \Psi_n(t) = n+\psi(t) & \text{ for } t\in[n,n+2],\\
  \Psi_n(t) = n+1 & \text{ for } t>n+2,
 \end{cases}
\end{equation*}
and $\psi:[0,2]\to[0,1]$ is smooth and increasing on $(0,2)$, such that $\psi(0)=0$, $\psi(2)=1$, $\psi'(0+) = 1$, $\psi'(2-)=0$, $\psi(t)\leq t$ for $t\in[0,2]$. We then use dominated convergence on the right-hand side and  monotone convergence on the left-hand side (note that by the Poincar\'e inequality $f$ is square-integrable). 
\end{remark}

We need another inequality introduced by Barthe and Kolesnikov in \cite{MR2438906}. Let $\tau\in (0,1)$, one says that a probability measure $\mu$ on $\RR^d$ satisfies the inequality $I(\tau)$ if there exists constants $B_1$ and $C_1$ such that for every smooth $f\colon\RR^d\to\RR$,
 \begin{equation}
 \label{eq:Itau}
  \Ent_\mu(f^2) \leq B_1 \int_{\RR^d} f^2 d\mu + C_1 \int_{\RR^d} |\nabla f|^2 \log^{1-\tau}\Bigl(e + \frac{f^2}{\int_{\RR^d} f^2 d\mu}\Bigr) d\mu.
 \end{equation}
This inequality is related to the previous ones. The next statement is a quotation of Theorem 4.1 in \cite{MR2438906}, with a stronger assumption (of a Poincar\'e inequality instead of a local Poincar\'e inequality)

\begin{theorem}[Barthe--Kolesnikov \cite{MR2438906}]
\label{th:barthe-kolesnikov}
Let $r\in(1,2)$. Let $\mu$ be a probability measure satisfying Inequality $I(2/r')$. Then $\mu$
satisfies a defective $F_r$-Sobolev inequality  
\eqref{eq:defective_FSob} and a defective modified log-Sobolev inequality with parameter $r$.

 If in addition $\mu$ satisfies a Poincar\'e inequality, then its satisfies an $F_r$-Sobolev inequality \eqref{eq:FSob} and a modified log-Sobolev inequality with parameter $r$, \eqref{eq:mLS}, with constants depending only on the constants of the input inequalities. 
\end{theorem}
We  establish a partial converse to the above implication:
\begin{theorem}
 \label{th:FSob_to_Itau} 
 Let $r\in(1,2)$.
 Assume that a probability measure $\mu$ on $\RR^d$ satisfies the  defective $F_r$-Sobolev inequality \eqref{eq:defective_FSob} with constants $B$ and $C$.
 Then $\mu$ satisfies the $I(2/r')$ inequality  \eqref{eq:Itau} with some constants $B_1$ and $C_1$ which depend only on $B$, $C$, and $r$.
\end{theorem}

\begin{proof} We reverse the reasoning from the proof of Theorem~4.1 in \cite{MR2438906}. Fix a smooth function $f$ such that the right-hand side of \eqref{eq:Itau} is finite.
We may
and do assume that $\int_{\RR^d} f^2|\ln(f^2)|d\mu<\infty$.\footnote{Indeed, like above let us define odd and non-decreasing functions $\Psi_n\colon\RR\to\RR$ by putting $\Psi_n(t)=t$ for $t\in[0,n)$, $\Psi_n(t) = n+1$ for $t>n+2$; for $t\in[n,n+2]$ let us take $\Psi_n(t) = n+\psi(t)$, where $\psi:[0,2]\to[0,1]$ is smooth and increasing on $(0,2)$, and satisfies $\psi(0)=0$, $\psi(2)=1$, $\psi'(0+) = 1$, $\psi'(2-)=0$, $\psi(t)\leq t$. Then the functions $f_n = \Psi_n(f)$ are smooth, bounded (and hence $\int_{\RR^d} f_n^2|\ln(f_n^2)|d\mu<\infty$) and converge to $f$ pointwise. After proving that \eqref{eq:Itau} holds for $f_n$, we obtain the assertion for $f$ by taking $n\to\infty$ and using monotone convergence on the left-hand side and the Lebesgue dominated convergence theorem on the right-hand side (note that we know that $f$ and $f_n$ are square-integrable, $|\nabla f_n |$ is up to a constant smaller than $|\nabla f|$, $f_n=f$ if $|f|\in[0,n]$, $|f_n|\leq |f|$ if $|f|\in[n,n+2]$, and if $|f|>n+2$, then  $\nabla f_n = 0$).} 
Consider the function
\begin{equation*}
 \Phi(x) = x^2 \log^{1-2/r'}(e+x^2), \quad x\in\RR,
\end{equation*}
(which is convex since the function $t\mapsto t\log^{1-2/r'}(e+t)$ is convex and increasing for $t>0$. Recall that $r'>2$).
Denote by $L$ the
Luxemburg 
norm of $f$ related to $\Phi$:
\begin{equation*}
 L = \inf\bigl\{ \lambda> 0 : \int_{\RR^d} \Phi(f/\lambda) d\mu \leq 1 \bigr\}.
\end{equation*}
Note that $L<\infty$, $\int_{\RR^d} \Phi(f/L) d\mu =1$ (by the definition of $L$), and $L^2\geq \int_{\RR^d} f^2 d\mu$ (since $\Phi(x)\geq x^2$).

 Set $g:= \sqrt{\Phi(f/L)}$. We have $\int_{\RR^d} g^2d\mu =1$ and~\eqref{eq:defective_FSob} reads
  \begin{equation}
  \label{eq:FSob_for_g}
  \int_{\RR^d} g^2 \bigl(\log^{2/r'}(1+g^2) - \log^{2/r'}(2) \bigr) d\mu
  \leq
  B \int_{\RR^d} g^2 d\mu +
  C \int_{\RR^d} |\nabla g|^2 d\mu.
 \end{equation} 
 Let us first express the right-hand side of this inequality in terms of $f$.
 For $x\in\RR$ denote $\varphi(x):= x\log^{1/2-1/r'}(e+x^2)$. 
 Then
 \begin{align*}
0\leq \varphi'(x) &=  \log^{1/2-1/r'}(e+x^2) + (1/2-1/q) \frac{2x^2}{e+x^2} \log^{-1/2-1/r'}(e+x^2)\\
&\leq 2 \log^{1/2-1/r'}(e+x^2)
\end{align*}
and thus
 \begin{align*}
  |\nabla g|^2 = \frac{|\nabla f|^2}{L^2} \bigl(\varphi'(f/L)\bigr)^2
  &\leq 4\frac{|\nabla f|^2}{L^2} \log^{1-2/r'}(e+f^2/L^2) \\
  &\leq  4\frac{|\nabla f|^2}{L^2}  \log^{1-2/r'}\Bigl(e+\frac{f^2}{\int_{\RR^d} f^2 d\mu} \Bigr).
 \end{align*}
Hence
\begin{equation}
 \label{eq:rhs_in_f}
   B \int_{\RR^d} g^2 d\mu +  C \int_{\RR^d} |\nabla g|^2 d\mu \leq B + 4C\int_{\RR^d}\frac{|\nabla f|^2}{L^2}  \log^{1-2/r'}\Bigl(e+\frac{f^2}{\int_{\RR^d} f^2 d\mu} \Bigr) d\mu.
\end{equation}

 As for the left-hand side of~\eqref{eq:FSob_for_g}, it is easy to see that there exists $\kappa_1=\kappa_1(r)>0$ such that, for $y>0$,
\begin{equation*}
 y\log^{1-2/r'}(e+y)\Bigl(\log^{2/r'}(1+y\log^{1-2/r'}(e+y))-\log^{2/r'}(2)\Bigr) \geq y\log(y) -\kappa_1. 
\end{equation*}
Applying this inequality with $y=f^2/L^2$, we arrive at
\begin{align*}
  \int_{\RR^d} g^2 \bigl(\log^{2/r'}(1+g^2) - \log^{2/r'}(2) \bigr) d\mu
 \geq
 \int_{\RR^d} \frac{f^2}{L^2}\log(f^2/L^2) d\mu - \kappa_1.
\end{align*}
Together with \eqref{eq:rhs_in_f} this yields
\begin{equation*}
 \int_{\RR^d} f^2\log(f^2/L^2) d\mu \leq (B+\kappa_1)L^2
 + 4C \int_{\RR^d} | \nabla f|^2\log^{1-2/r'}\Bigl(e+\frac{f^2}{\int_{\RR^d} f^2 d\mu} \Bigr) d\mu.
\end{equation*}
It remains to replace the expression on the left-hand side by $\Ent_{\mu}(f^2)$ and estimate $L^2$.

Since
\begin{align*}
 \Ent_{\mu}(f^2) &= \inf_{t>0} \int_{\RR^d} \Bigl(f^2\log(f^2/t) -f^2 + t\Bigr)  d\mu\\
 &\leq  \int_{\RR^d} \Bigl( f^2\log(f^2/L^2) -f^2 + L^2\Bigr) d\mu,
\end{align*}
we conclude that
\begin{equation}
\label{eq:up_to_L2}
 \Ent_{\mu}(f^2) \leq (B+\kappa_1+1) L^2
+  4C \int_{\RR^d} |\nabla f|^2\log^{1-2/r'}\Bigl(e+\frac{f^2}{\int_{\RR^d} f^2 d\mu} \Bigr) d\mu.
\end{equation}

Finally, it is easy to see that for every $\varepsilon >0$ there exist $\kappa_2 = \kappa_2(\varepsilon,r)$ such that, for $y>0$,
\begin{align*}
 y\log^{1-2/r'}(e+y)
 \leq \varepsilon y \log(y) + \kappa_2 
 \end{align*}
Using first the definition of $L$ and the fact that $L^2\geq \int_{\RR^d} f^2 d\mu$, and then the above bound (with $y=f^2/\int_{\RR^d} f^2 d\mu$) we can thus estimate
\begin{align*}
L^2  &= 
\int_{\RR^d} f^2 \log^{1-2/r'}(e+f^2/L^2) d\mu
\leq \int_{\RR^d} f^2 \log^{1-2/r'}\Bigl(e+\frac{f^2}{\int_{\RR^d} f^2 d\mu} \Bigr) d\mu\\
&\leq \varepsilon \Ent_\mu(f^2) + \kappa_2\int_{\RR^d} f^2 d\mu.
\end{align*}
Eventually, for $\varepsilon$ small enough we combine this bound with~\eqref{eq:up_to_L2} and
simplify the entropy terms (recall that by our assumption $\Ent_\mu(f^2)<\infty$) in order to reach  the claim. 
\end{proof}

\section{Proof of the main result and its corollaries}
\label{sec:proof-of-main-and-cor}

\begin{proof}[Proof of Theorem~\ref{thm:main:Toulouse}.]
Our assumption is that $\mu$ satisfies a Lata{\l}a--Oleszkiwicz inequality with parameter $r$.
Therefore by Lemma \ref{lem:LO-to-Poincare} it also satisfies a Poincar\'e inequality, and 
by Theorem \ref{thm:LO-equiv-forms}  (and Remark~\ref{rem:tutti-frutti} if $\mu$ is not absolutely continuous) it satisfies a (tight) $F_r$-Sobolev inequality. From Theorem \ref{th:FSob_to_Itau}, we deduce that $\mu$ enjoys an $I(2/r')$-inequality. Eventually Theorem \ref{th:barthe-kolesnikov} asserts that the $I(2/r')$, together with the Poincar\'e inequality, implies a (tight) modified log-Sobolev inequality with parameter $r$. The constants that we obtain in the above inequalities only depend on $r$ and $C_{LO(r)}$. Proving  the claimed dependence in $C_{LO(r)}$ is straightforward, but requires to track the constants in the various intermediate statements. We omit the details. 
\end{proof}

\begin{remark}
Let us comment here that for $d=1$ it is known that the inequalities \eqref{eq:LO} and \eqref{eq:mLS} hold if and only if they hold with the integration  with respect to $\mu$ on the right-hand side replaced by integration with respect to $\mu_{ac}$, the absolutely continuous part of $\mu$ (cf. \cite{MR1682772}, Appendix of \cite{MR2464097}, Appendix of \cite{GozW2}).
\end{remark}

For the proofs of the corollaries we need one more technical lemma. We denote by $ H_{r'}^*(t)\coloneqq \sup_{s\in\RR} \{st-H_{r'}(s)\}$, $t\in\RR$, the Legendre transform of $H_{r'}$ (we refer to the book \cite{MR0274683} for more information on this topic).

\begin{lemma}\label{lem:H-H*} Let $r\in(1,2)$. The function $ H_{r'}^*$ is given by the formula
		\begin{align*}
		 H_{r'}^*(t)  =
		\begin{cases}
		t^2/4 &\text{if } 0\leq |t|\leq 2,\\
		|t|-1 &\text{if } 2\leq |t| \leq r',\\
		\frac{1}{r-1} (\frac{1}{r'}|t|)^{r} & \text{if } |t| \geq r'.
		\end{cases}
		\end{align*}
Moreover, $ H_{r'}^*(t) \geq \frac{1}{4} \min\{t^2, |t|^r \} $.
\end{lemma}

\begin{proof}[Proof.] The first part is a straightforward calculation.  To prove the second part, first notice that
\begin{equation*}
\inf_{r\in(1,2)}\frac{1}{r-1} \Big(\frac{1}{r'}\Big)^{r}  = 1/4.
\end{equation*}
This allows to verifty the  inequality  $H_{r'}^*(t) \geq \frac{1}{4} \min\{t^2, |t|^r\}$ for $t\in[0,2]\cup[r',\infty)$. It remains to prove the inequality on the interval $[2,r']$, where  $H_{r'}^*(t)=t-1$ is affine in $t$, and $\frac{1}{4} \min\{t^2, |t|^r\}=\frac14 t^r$ is convex in $t$. 
Thefore it is enough to check the inequality at the endpoints of the interval, which we have already done. 
\end{proof}

\begin{proof}[Proof of Corollary~\ref{cor:Toul-conc-functions}.]
A classical argument of Herbst (see, e.g., \cite{MR1849347}) allows to deduce concentration bounds from log-Sobolev inequalities. It was implemented in \cite{MR2430612} for modified log-Sobolev inequalities with energy terms $H(\nabla f/f)$ involving general functions $H$. We rather follow the calculation of~\cite{MR3456588} which is more suited to the case $H=H_{r'}$.  Take a function $f\colon (\RR^d)^n\to\RR$ and denote
  \begin{equation*}
     A = \sup_{x\in\RR^{dn}} |\nabla f(x)|,\qquad  B = \sup_{x\in\RR^{dn}} \|\nabla f(x)\|_{r',2}.
   \end{equation*}
Moreover, let $F(\lambda) = \int_{\RR^{dn}} e^{\lambda f(x)} d\mu^{\otimes n}$. Then
\begin{equation*}
\lambda F'(\lambda) = \int_{\RR^{dn}} \lambda  f(x)e^{\lambda f(x)} d\mu^{\otimes n}
\end{equation*}
and hence, since $\mu$ satisfies the modified log-Sobolev inequality with some constant $C=C(C_{LO},r)$ (by Theorem~\ref{thm:main:Toulouse}) and by the tensorization property,
\begin{align*}
\lambda F'(\lambda) -  F(\lambda)\log  F(\lambda)
&= \Ent_{\mu^{\otimes n}}( e^{\lambda f})\\
&\leq C \int_{\RR^{dn}} \sum_{i=1}^n H_{r'}\Bigl( \frac{\lambda}{2}|\nabla_i f| \Bigr) e^{\lambda f} d\mu^{\otimes n }\\
&\leq 2 C \max\bigl\{(A\lambda/2)^2, (B\lambda/2)^{r'}\bigr\} F(\lambda),
\end{align*}
where we used the inequality $\sum_i \max\{a_i^2, b_i^{r'}\} \leq 2 \max\{\sum_i a_i^2, \sum_i b_i^{r'}\}$.
After dividing both sides by $\lambda^2 F(\lambda)$ we can rewrite this as
\begin{equation*}
\Big(\frac{1}{\lambda} \log  F(\lambda)\Big)' \leq 2 C \max\bigl\{(A\lambda/2)^2, (B\lambda/2)^{r'}\bigr\}/\lambda^2.
\end{equation*}
Since  the right-hand side is an increasing function of $\lambda>0$ and
\begin{equation*}
 \lim_{\lambda\to 0^+} \frac{1}{\lambda} \log  F(\lambda) = \int_{\RR^{dn}} f d\mu^{\otimes n},
\end{equation*}
 we deduce from the last inequality that
\begin{equation*}
\frac{1}{\lambda} \log  F(\lambda) \leq \int_{\RR^{dn}} f d\mu^{\otimes n} + 2C \max\bigl\{(A\lambda/2)^2, (B\lambda/2)^{r'}\bigr\}/\lambda,
\end{equation*}
which is equivalent to
\begin{equation*}
\int_{\RR^{dn}} e^{\lambda f}d\mu^{\otimes n} \leq \exp\Bigl(\lambda\int_{\RR^{dn}} f d\mu^{\otimes n} +  2C \max\bigl\{(A\lambda/2)^2, (B\lambda/2)^{r'}\bigr\}\Bigr).
\end{equation*}

Therefore from Chebyshev's inequality we get, for $t>0$ and any $\lambda>0$, 
\begin{align*}
\mu^{\otimes n}\Bigl(f \geq \int_{\RR^{dn}} f d\mu^{\otimes n} +t\Bigr)
&\leq \frac{\int_{\RR^{dn}} e^{2\lambda f} d\mu^{\otimes n}}{\exp(2\lambda\int_{\RR^{dn}} f d\mu^{\otimes n} + 2\lambda t)}\\
&\leq  \exp\bigl( - 2\lambda t + 2C \max\bigl\{(A\lambda)^2, (B\lambda)^{r'}\bigr\}\bigr).
\end{align*}
Now we can optimize the right-hand side with respect to $\lambda$.  Let $U$ and $V$ be such that $A = U^{1/2}V$, $B = U^{1/r'}V$. We have
\begin{equation*}
\max\bigl\{(A\lambda)^2, (B\lambda)^{r'}\bigr\}\bigr) = U\max\{(V\lambda)^2,(V\lambda)^{r'}\} = U H_{r'}(V\lambda)
\end{equation*}
and hence
\begin{equation*}
\mu^{\otimes n}\Bigl(f \geq \int_{\RR^{dn}} f d\mu^{\otimes n} +t\Bigr)
\leq  \exp\bigl( -2CU H_{r'}^* \bigl(\frac{t}{CUV}\bigr)\bigr).
\end{equation*}
 Using Lemma~\ref{lem:H-H*} and the definitions of $U$ and $V$  we get
\begin{align*}
\mu^{\otimes n}\Bigl(f \geq \int_{\RR^{dn}} f d\mu^{\otimes n} +t\Bigr)
\leq  \exp\Big(- \frac{1}{2} \min\Big\{\frac{t^2}{CA^2}, \frac{t^{r}}{ C^{r-1} B^{r}} \Big\}\Big),
\end{align*}
which yields the assertion of the corollary.
\end{proof}

\begin{proof}[Proof of Corollary~\ref{cor:Toul-conc-functions-Lip}]
Let $f_{\varepsilon}$ be the convolution of $f$ and a Gaussian kernel, i.e., $f_{\varepsilon}(x) = \EE f(x+\sqrt{\varepsilon}G)$, where $G\sim\mathcal{N}(0,I)$. This function clearly inherits from $f$ the estimates of the Lipschitz constants. Since it is smooth, the $\ell_2$-norm and the norm $\|\cdot\|_{r',2}$ of its gradient can be estimated pointwise by $L_2$ and $L_{r,2}$, respectively. Therefore we can apply Corollary~\ref{cor:Toul-conc-functions} to $f_\varepsilon$.
Moreover, $|f_{\varepsilon}(x)-f(x)| \leq L_2\sqrt{\varepsilon} \EE |G|$ and hence $f_{\varepsilon}$ converges uniformly to $f$ as $\varepsilon$ tends to zero. This observation ends the proof of the corollary.
\end{proof}

\begin{proof}[Proof of Corollary~\ref{cor:Toul-conc-set}.]
We follow the approach of Bobkov and Ledoux from Section 2 of \cite{MR1440138}.
Take a set $A\subset \RR^{dn}$ with $\mu^{\otimes n} (A)\geq 1/2$. For $x=(x_1,\ldots, x_n)\in(\RR^d)^n$, denote
\begin{equation*}
 F(x) =F_A(x) = \inf_{a\in A} \sum_{i=1}^n \min\{|x_i-a_i|^2, |x_i-a_i|^r\}
\end{equation*}
(note that here $|\cdot|$ is the $\ell_2$-norm on $\RR^d$). Take any $t>0$ and set $f = \min\{F,t\}$.
We claim that for all $x,y\in(\RR^d)^n$,
\begin{equation}\label{eq:claim_ala_BL}
 |f(x) - f(y)| \leq 2\sqrt{t} |x-y|, \qquad |f(x) - f(y)| \leq 2 t^{1/r'} \|x-y\|_{r,2}.
\end{equation}

Suppose that we already know that this holds. 
Note that $(2 t^{1/r'})^r= 2^r t^{r-1}$. 
 Also, $\int_{\RR^{dn}} f d\mu^{\otimes n} \leq t/2$ since $F=0$ on $A$ and $\mu^{\otimes n}(A)\geq 1/2$.
Consequently, by Corollary~\ref{cor:Toul-conc-functions-Lip}  and~\eqref{eq:claim_ala_BL}, 
\begin{align*}
\mu^{\otimes n} (F_A \geq t) &\leq \mu^{\otimes n} (f\geq t) \leq \mu^{\otimes}\bigl(f\geq \int_{\RR^{dn}} f d\mu^{\otimes n} +t/2\bigr)\\
&\leq
\exp\Bigl(-\frac{1}{2} \min\Bigl\{\frac{(t/2)^2}{ 4tC}, \frac{(t/2)^r}{2^rt^{r-1}C^{r-1}}\Bigr\}\Bigr)\\
&=  \exp\Bigl(-\frac{t}{2} \min\Bigl\{ \frac{1}{16 C},\frac{1}{4^r C^{r-1}}\Bigr\} \Bigr)\le \exp(-Kt),
\end{align*}
where $K = \frac{1}{32}\min\{1/C, 1/{C^{r-1}}\}$ and $C$ is the constant with which, by Theorem~\ref{thm:main:Toulouse}, the modified log-Sobolev inequality holds for $\mu$. Since clearly
\begin{equation*}
 \{ F_A  < t\} \subset A +\Bigl\{(x_1,\ldots,x_n)\in(\RR^d)^n : \sum_{i=1}^n \min\{|x_i|^2,|x_i|^r\} \leq t\Bigr\},
\end{equation*}
this yields the first assertion of the corollary. The second part follows by the inclusion
\begin{equation*}
 \Bigl\{(x_1,\ldots,x_n)\in(\RR^d)^n : \sum_{i=1}^n \min\{|x_i|^2,|x_i|^r\} \leq t\Bigr\} \subset \sqrt{t} B_2^{dn} + t^{1/r} B_{r,2}^{n,d}.
\end{equation*}

It remains to prove the claim~\eqref{eq:claim_ala_BL}. To this end, consider the functions
\begin{equation*}
 G(x) = \sum_{i=1}^n \min\{|x_i|^2, |x_i|^r\}
\end{equation*}
and $g(x) = \min\{G,t\}$. Since $g$ is locally Lipschitz it suffices to show that, a.e.,
\begin{equation*}
 \sum_{i=1}^n |\nabla_i g|^2 \leq 4t, \qquad 
 \sum_{i=1}^n |\nabla_i g|^{r'} \leq 2^{r'} t.
\end{equation*}
Indeed, this will imply that \eqref{eq:claim_ala_BL} holds with $g$ in place of $f$ (note that the norm $\|\cdot\|_{r,2}$ is dual to the norm $\|\cdot\|_{r',2}$). Since $f(x) = \inf_{a\in A} g(x-a)$ (and the infimum of Lipschitz functions is Lipschitz with the same constant),  the same estimates will be inherited by $f$.

On the open set $\{G>t\}$ the estimates obviously hold, since $g$ is constant. The set $\{G=t\}$ is Lebesgue negligible.  Thus in what follows it suffices to consider the set $\{G<t\}$ on which $g=G$.

If, for some $i$, $|x_i|< 1$, then
\begin{align*}
 |\nabla_i g(x) |^2 &= 4 |x_i |^2 = 4 \min\{|x_i|^2, |x_i|^r\},
\\
 |\nabla_i g(x) |^{r'} &= 2^{r'} |x_i |^{r'} \leq  2^{r'} |x_i|^2 = 2^{r'} \min\{|x_i|^2, |x_i|^r\}.
\end{align*}
If on the other hand $|x_i|> 1$, then 
\begin{align*}
|\nabla_i g(x) |^2 &= r^2 |x_i |^{2(r-1)} 
                     \leq 4 |x_i |^{r} 
                     = 4 \min\{|x_i|^2, |x_i|^r\},\\
|\nabla_i g(x) |^{r'} &= r^{r'} |x_i|^r=r^{r'}  \min\{|x_i|^2, |x_i|^r\}.
\end{align*}
Thus, a.e. (the set where $|x_i|=1$ for some $i$ is negligible),
\begin{align*}
|\nabla_i g(x) |^2 &\leq 4 \min\{|x_i|^2, |x_i|^r\},\\
|\nabla_i g(x) |^{r'} &\leq 2^{r'}  \min\{|x_i|^2, |x_i|^r\}.
\end{align*}
Consequently, on the set $\{G<t\}$, it holds a.e.
\begin{align*}
  \sum_{i=1}^n |\nabla_i g(x)|^2 &\leq  
  4 G(x) \leq 4t,\\
    \sum_{i=1}^n |\nabla_i g(x)|^{r'} &\leq  
    2^{r'} G(x) \leq 2^{r'} t.
\end{align*}
Therefore the proof is complete.
\end{proof}

\section{Criteria for measures on the real line}
\label{sec:R}
From now on we restrict to probability measures on the real line. In this setting, more tools 
are available. For several functional inequalities, workable equivalent criteria are available. They
are based on Hardy type inequalities, of the form 
\[ \int_{\RR^+} |f-f(0)|^p d\mu\le A \int_{\RR^+} |f'|^p d\nu.\]
 We refer to  \cite{BOOK-ane-et-al} for the history 
of the topic, from the original book of Hardy, Littlewood and P\'olya, to the general version by
Muckenhoupt. The textbook  \cite{BOOK-ane-et-al}  also mentions that such Hardy inequalities yield the following criterion for Poincar\'e inequalities on $\RR$ (where we include a numerical improvement from
\cite{MR2464097}):

Let $\mu$ be a probability measure on $\RR$, with median $m$. Let $\nu$ be a probability 
measure on $\RR$, and let $n$ denote the density of its absolutely continuous part. Then
the (possibly infinite) best constant $C_P$ such that for all smooth $f$,
\[ \Var_\mu(f)\le  C_P\int_{\RR} (f')^2 d\nu\]
verifies $\max(B_P^+,B_P^-)\le C_P \le 4\max(B_P^+,B_P^-) $, with 
\begin{equation}\label{eq:BP}
 B_P^+=\sup_{x>m} \mu([x,+\infty)) \int_m^x \frac1n, \qquad
   B_P^-=\sup_{x<m} \mu((-\infty,x]) \int_x^m \frac1n,
 \end{equation}
   where by convention $0\cdot \infty=0$.
   
   Bobkov and G\"otze \cite{MR1682772} extended the reach of these methods, by proving a similar statement for log-Sobolev inequalities of the form 
   \[ \Ent_\mu(f^2)\le C_{LS} \int_{\RR}  (f')^2 d\nu.\]
   Their result reads as the previous one, with different numerical constants and $B_P^+$, 
   $B_P^-$ replaced by     
   \begin{equation}
   \label{eq:BLS}
    B_{LS}^+=\sup_{x>m} \mu([x,+\infty)) \log\left( \frac{1}{\mu([x,+\infty))}\right)\int_m^x \frac1n,
    \end{equation}
and $B_{LS}^-$ defined similarly for $x<m$. 
 This criterion was later extended to the 
 Lata{\l}a--Oleszkiewicz inequality~\eqref{eq:LO}. 
Let $\mu$ be a probability measure on $\RR$.  Denote by $m$ the median of $\mu$ and by $n$ the density of its absolutely continuous part. Barthe and Roberto \cite{MR2052235} proved that $\mu$ satisfies the Lata{\l}a--Oleszkiewicz inequality~\eqref{eq:LO} if and only if $ \max\{ B_{LO(r)}^+, B_{LO(r)}^-\} <\infty$,
where
\begin{equation}
\label{eq:LO-crit}
 B_{LO(r)}^+ \coloneqq
 \sup_{x>m} \mu([x,\infty)) \log^{2/r'} \Bigl(1+\frac{1}{2\mu([x,\infty))}\Bigr) \int_m^x \frac{1}{n(t)} dt 
\end{equation}
and $B_{LO(r)}^-$ is defined similarly but with $x<m$. Moreover the best possible constant $C_{LO(r)}$ in~\eqref{eq:LO} is comparable to  $\max\{ B_{LO(r)}^+, B_{LO(r)}^-\}$, up to numerical constants which do not depend on $r\in(1,2)$.

In the subsequent paper  \cite{MR2430612}, Barthe and Roberto provided a criterion for the modified log-Sobolev inequality~\eqref{eq:mLS}. However, they did not reach a full equivalence.
Here is the outline of Theorem 10 in  \cite{MR2430612}. Let $d\mu(t)=n(t)dt$ be a probability
measure on $\RR$ with median $m$.
 If $\mu$ satisfies the Poincar\'e inequality with constant $C_P$ and $\max\{ B_{mLS(r)}^+, B_{mLS(r)}^- \} < \infty$, where
\begin{equation}
\label{eq:BmLS}
 B_{mLS(r)}^+ \coloneqq
 \sup_{x>m} \mu([x,\infty)) \log \Bigl(\frac{1}{\mu([x,\infty))}\Bigr) \Bigl(\int_m^x \frac{1}{n(t)^{r-1}} dt \Bigr)^{1/(r-1)}
\end{equation}
and $B_{mLS}^-$ is defined similarly but with $x<m$, then $\mu$ satisfies the modified log-Sobolev inequality~\eqref{eq:mLS} with constant
\begin{equation*}
C_{mLS} \leq 235C_P + 2^{r'+1}\max\{B_{mLS(r)}^+, B_{mLS(r)}^-\}. 
\end{equation*}
The converse implication is, so far, known only under the following additional assumption: there 
exists $\varepsilon>0$ such that for all $x\neq m$
\begin{equation}
\label{eq:hyp-crit-mLS}
 \frac{1}{n(x)^{r-1}}\ge \varepsilon  \int_{\min(m,x)}^{\max(m,x)} \frac{1}{n^{r-1}}\cdot
 \end{equation}
In this case, if 
$\mu$ satisfies the modified log-Sobolev inequality~\eqref{eq:mLS}, then
 \begin{equation*}
 \max\{B_{mLS(r)}^+, B_{mLS(r)}^-\}<\infty
 \end{equation*}
 and this quantity can be estimated in terms of the constant $C_{mLS(r)}$ up to constants depending on $r$ and $\varepsilon$. The Poincar\'e inequality is a classical consequence 
 of modified log-Sobolev inequality, exactly as in Lemma \ref{lem:LO-to-Poincare}.

Even though the above criteria involve simple concrete quantities, it does not seem easy to use 
them in order to reprove our main result Theorem \ref{thm:main:Toulouse} for measures on $\RR$.
 However, if one assumes for example that $d\mu(x) = \exp(-V(x)) dx$, $x\in\RR$, where $V$ is symmetric, of class $C^2$,
$ \liminf_{x\to\infty} V'(x) > 0$,
and
\begin{equation*}
 \lim_{x\to\infty} \frac{V''(x)}{V'(x)^2} =0,
\end{equation*}
then one can estimate the quantities $B^{+}_{LO(r)}, B^+_{mLS(r)}$ and show that the Lata{\l}a--Oleszkiewicz inequality~\eqref{eq:LO} is equivalent to the modified log-Sobolev inequality~\eqref{eq:mLS} and furthermore to the condition
\begin{equation}
\label{eq:workable_BR-crit}
 \limsup_{x\to\infty} \frac{V(x)}{V'(x)^{r'}} < \infty
\end{equation}
(see  Remark~21 in \cite{MR2430612}).

In the rest of the paper we use and develop one-dimensional criteria in order to study whether the modified log-Sobolev inequality is actually equivalent to other inequalities which are known to imply it.

\section{Weighted vs. modified log-Sobolev inequality}
\label{sec:weighted}

It is known that if a probability measure $\mu$ on $\RR^d$ satisfies a certain weighted log-Sobolev inequality (and an integrability condition), then it also satisfies a modified log-Sobolev inequality, see Theorem~3.4 in~\cite{MR3008255} (in the context of a specific measure on the real line a similar argument appears already  in the large entropy case of the proof of Theorem~3.1 from \cite{MR2198019}). The goal of this subsection is to show that the converse implication does not hold in general, even for measures on the real line.

First we present a workable criterion for the weighted log-Sobolev inequality.

\begin{proposition}
 \label{prop:char_weighted}
 Let $d\mu(x)=e^{-V(x)} dx$ be a probability measure on the real line. Let  $V\colon\RR\to\RR$ be  even and locally bounded. Assume that  in some neighborhood of $\infty$,  the function $V$ is of class $C^2$, and that
 \begin{enumerate}[(i)]
 \item $\liminf_{x\to\infty} V'(x) >0$,
\item 
  $ \lim_{x\to\infty} \frac{V''(x)}{V'(x)^2} =0$.
 \end{enumerate}
Then, there exists $C<\infty$ such that $\mu$ satisfies the following weighted log-Sobolev inequality: for every $f\colon\RR\to\RR$,
\begin{equation}
\label{eq:weighted}
 \Ent_{\mu}(f^2) \leq C \int_{\RR} f'(x)^2(1+|x|^{2-r}) d\mu(x),
\end{equation}
if and only if
\begin{equation*}
\limsup_{x\to\infty} \frac{V(x)}{|x|^{2-r} V'(x)^2} <\infty.
\end{equation*}
\end{proposition}

\begin{remark}
 Condition $(ii)$ can be weakened to $ \limsup_{x\to\infty} \frac{|V''(x)|}{V'(x)^2} <1$.
\end{remark}

\begin{proof}[Proof of Proposition~\ref{prop:char_weighted}]
Denote $W(x):=V(x) - \log(1+|x|^{2-r})$, $x\in\RR$.
By the Bobkov--G\"otze criterion \cite{MR1682772} (see \eqref{eq:BLS}),
$\mu$ satisfies the weighted log-Sobolev inequality if and only if
\begin{equation}
\label{eq:BG_for_weighted}
 \sup_{x>0} \mu((x,\infty))\log\Bigl( \frac{1}{\mu((x,\infty))} \Bigr) \int_0^x e^{W(t)} dt < \infty.
\end{equation}
Of course, it suffices to investigate what happens for $x\to\infty$.
Note that
\begin{equation*}
\liminf_{x\to\infty} W'(x) = \liminf_{x\to\infty}\Bigl( V'(x) - \frac{(2-r)x^{1-r}}{1+x^{2-r}} \Bigr) > 0
\end{equation*}
(by Assumption (i)) and 
\begin{equation*}
\lim_{x\to\infty} \frac{W''(x)}{W'(x)^2} = 0
\end{equation*}
(by (ii) and the fact that $W''(x) = V''(x) + o(1)$). Thus, as $x\to\infty$,
\begin{align*}
 \mu((x,\infty)) = \int_x^{\infty} e^{-V(t)} dt &\sim \frac{e^{-V(x)}}{V'(x)},
 \\
 \int_0^x e^{W(t)} dt &\sim  \frac{e^{W(x)}}{W'(x)} = \frac{e^{V(x)}}{(1+|x|^{2-r})(V'(x)+o(1))}
\end{align*}
(here by `$\sim$' we mean that the ratio of both sides tends to $1$ as $x\to\infty$; to prove that this is indeed the case it  suffices to consider the ratio of the derivatives of both sides). Therefore, \eqref{eq:BG_for_weighted} holds if and only if
\begin{equation*}
 \limsup_{x\to\infty} \frac{V(x) + \log V'(x)}{(1+|x|^{2-r})(V'(x)+o(1))V'(x)} <\infty,
\end{equation*}
which, since $V'(x)$ is bounded away from zero as $x\to\infty$, happens if and only if
\begin{equation*}
\limsup_{x\to\infty} \frac{V(x)}{|x|^{2-r} V'(x)^2} <\infty. \qedhere
\end{equation*} 
\end{proof}

Our example is a modification of the example constructed by Cattiaux and Guillin \cite{MR2257848} to prove that the log-Sobolev inequality is strictly stronger than Talagrand's transportation cost inequality.

\begin{proposition}
 \label{prop:counterex_weighted}
For $r\in(1,2)$ and $ \max\{r/2,r-1/r\} < \beta -1 <  r-1/2$ define
\begin{equation*}
 U(x) =
 U_{r,\beta}(x) = |x|^{r+1} + (r+1) |x|^r \sin^2(x) + |x|^\beta,\quad x\in\RR.
\end{equation*}
Let $\mu_{r,\beta}$ be the probability measure with density proportional to $e^{-U_{r,\beta}(x)}$. Then $\mu_{r,\beta}$ satisfies the modified log-Sobolev inequality \eqref{eq:mLS} and the Lata\l{}a--Oleszkiewicz inequality \eqref{eq:LO} (with $d=1$).

On the other hand, $\mu_{r,\beta}$ does not satisfy the weighted log-Sobolev inequality~\eqref{eq:weighted}.
\end{proposition}

\begin{proof}
Let us first note that $\beta\in(r,r+1)$. For $x>0$,
 \begin{align*}
 U(x) &= x^{r+1} + (r+1) x^r \sin^2(x) + x^\beta, \\
 U'(x)
  &=  (r+1)(1+\sin(2x))x^r + (r+1)rx^{r-1}\sin^2(x) + \beta x^{\beta-1}.  
 \end{align*}
Clearly, $U'(x) \geq \beta x^{\beta-1}$; in particular $\liminf_{x\to\infty} U'(x) >0$. Moreover, for $x>1$, $|U''(x)|$ can be bounded by $M x^r$ for some constant $M=M(r,\beta)$. Thus,
\begin{equation*}
 \lim_{x\to\infty} \frac{|U''(x)|}{U'(x)^2} \leq \lim_{x\to\infty} \frac{M x^r}{\beta^2 x^{2(\beta-1)}} = 0,
\end{equation*}
since $\beta -1 > r/2$. We are thus in position to apply workable versions of the criteria for the modified and weighted log-Sobolev inequalities (note that the normalization of $\mu_{r,\beta}$ amounts to adding a constant to the potential $U$, which does not affect the calculations and reasoning below).

First note that
\begin{equation*}
 \lim_{x\to\infty} \frac{U(x)}{U'(x)^{r'}} \leq \lim_{x\to\infty} \frac{(r+3) x^{r+1}}{\beta^{r'} x^{(\beta-1)r'}} = 0,
\end{equation*}
since $(\beta -1)r' > (r-1/r)r' = r+1$. Thus, by the Barthe--Roberto criterion (see \eqref{eq:workable_BR-crit}), $\mu_{r,\beta}$ satisfies the modified log-Sobolev and the Lata\l{}a--Oleszkiewicz inequality.

On the other hand, for certain values of $x\to\infty$ (e.g., for $x=k\pi -\pi/4$, $k\in\NN$), we have $|U'(x)| \leq ((r+1)r+\beta) x^{\beta-1}$. Hence
\begin{equation*}
 \limsup_{x\to\infty} \frac{U(x)}{x^{2-r} U'(x)^2} \geq \lim_{x\to\infty} \frac{x^{r+1}}{x^{2-r} ((r+1)r+\beta)^2 x^{2(\beta-1)}} = \infty,
\end{equation*}
 since $\beta-1 < r-1/2$. Thus, by Proposition~\ref{prop:char_weighted} above, $\mu_{r,\beta}$ cannot satisfy the weighted log-Sobolev inequality.
\end{proof}

\begin{remark}
The introduction of \cite{MR2520723}, suggests that the results of our Theorem~\ref{thm:main:Toulouse} are contained in \cite{MR2446080}, namely that it follows from \cite{MR2446080} that the $F_r$-Sobolev inequality~\eqref{eq:FSob} implies the modified log-Sobolev inequality~\eqref{eq:mLS}. We would like to rectify this: Wang's paper \cite{MR2446080} deals with measures with \emph{faster} decay than Gaussian. He proves that in that setting an appropriate super Poincar\'e inequality (or equivalently, an appropriate $F$-Sobolev inequality) implies a certain weighted log-Sobolev inequality. However, in our setting (measures with tail decay \emph{slower} than Gaussian), we have an example of a measure which satisfies the modified log-Sobolev inequality~\eqref{eq:mLS} and the Lata\l{}a--Oleszkiewicz inequality~\eqref{eq:LO} (or equivalently, the $F_r$-Sobolev inequality~\eqref{eq:FSob}), but does not satisfy the weighted log-Sobolev inequality~\eqref{eq:weighted}. Therefore Theorem~\ref{thm:main:Toulouse} cannot be deduced from Wang's paper \cite{MR2446080}.
\end{remark}


\section{On potentials with vanishing derivatives}

\subsection{Motivation}

Recall that $r\in(1,2)$ is the parameter associated with the Lata{\l}a--Oleszkiewicz inequality~\eqref{eq:LO} and the modified log-Sobolev inequality~\eqref{eq:mLS}.
Throughout this section we consider symmetric probability measures on the real line of the form
\[
 d\mu(x) = d\mu_V(x) = \frac{1}{Z} \exp(-V(x)) dx, \quad x\in\RR,
\]
where $V\colon\RR\to\RR$ is even and $Z$ is the normalization constant.

It is easy to see that if $\varepsilon\in[0,1)$ and for $x\in \RR$
\begin{equation*}
V(x) = \big|x+\varepsilon\sin(x)\big|^r, \quad x>0,
\end{equation*}
then $\mu_V$ satisfies both the Lata{\l}a--Oleszkiewicz inequality~\eqref{eq:LO} and the modified log-Sobolev inequality~\eqref{eq:mLS}.
Indeed, if $\varepsilon\in[0,1)$, then $\liminf_{x\to+\infty} V'(x)>0$, $\lim_{x\to\infty} V''(x)/V'(x)^2=0$ and the claim follows  from the simplified versions of the Barthe--Roberto criteria (see~\eqref{eq:workable_BR-crit}).

This example becomes more interesting for $\varepsilon=1$: since for any integer $k$, $V'((2k+1)\pi)=0$  we cannot apply the simplified asymptotic versions of the criteria. In particular, one
would like to know if, for  measures with such potentials,  the modified log-Sobolev inequality~\eqref{eq:mLS} and the  Lata\l{}a--Oleszkiewicz inequality~\eqref{eq:LO} are valid simultaneously.

In the limit case $r=2$,  Cattiaux \cite{MR2188585} proved   that if 
\[
 V(x) = x^2+2\lambda x\sin(x), \quad x>0,
\]
then $\mu_V$ satisfies the classical log-Sobolev inequality if and only if $|\lambda|<1$ (note that this potential differs from $(x+\lambda\sin(x))^2$ only by a bounded perturbation). He used probabilistic methods which seem to rely on the fact that $r=2$. Below we present an analytic approach and obtain an extension of his results.

\subsection{Results}

For $\alpha > 1$ define
\begin{equation*}
 V_{\alpha}(x) = \big|x+ \sin(x)\big|^\alpha,\quad x\in \RR.
\end{equation*}
Let $\nu_{\alpha}$ be the probability measure
with density proportional to $\exp(-V_{\alpha})$:
\[
 d\nu_{\alpha}(x) = \frac{1}{Z_\alpha} \exp(-V_\alpha(x))dx, \quad x\in\RR.
\]

\begin{proposition}
\label{prop:vanishing-V-prim}
Let $\alpha>1$ and $r\in(1,2)$. The following assertions are equivalent 
\begin{enumerate}[(i)]
	\item $r\le r_0(\alpha):= \frac{3\alpha}{2\alpha+1}$,
	\item  $\nu_{\alpha}$ satisfies
the Lata{\l}a--Oleszkiewicz inequality~\eqref{eq:LO} with parameter $r$,
   \item  $\nu_{\alpha}$ satisfies the modified log-Sobolev inequality~\eqref{eq:mLS}   with parameter $r$.
\end{enumerate}
\end{proposition}

\begin{remark}
For  Cattiaux's example the threshold is $r_0(2)=6/5$.
\end{remark}

The  threshold $r_0(\alpha)$ in Inequalities \eqref{eq:LO} and \eqref{eq:mLS} suggests a weaker concentration than the one actually exhibited by the measures $\nu_{\alpha}$, which 
 is better described by transportation cost inequalities, see \cite{Tal,TalConcMeasure}.
Let $\alpha\in(1,2]$. Recall that we say that a probability measure $\mu$ on the real line satisfies the transport--entropy inequality $\mathbf{T}_{\min\{x^2, |x|^\alpha\}}(a)$  if for any probability measure $\sigma$ on the real line
\[
 \mathcal{T}_{\alpha,a} (\mu,\sigma) \leq H(\sigma|\mu),
\]
where $\mathcal{T}_{\alpha,a}$ is the optimal transport cost between the measures $\mu$ and $\sigma$ with respect to the cost function $t\mapsto \min\{ (at)^2, |at|^\alpha\}$, i.e.,
\[
\mathcal{T}_{\alpha,a} (\mu,\sigma) = \inf_\pi \int_\RR\int_\RR \min\{ (a(x-y))^2, |a(x-y)|^\alpha\} d\pi(x,y),
\]
where the infimum runs over the set of couplings between $\mu$ and $\sigma$, and $H(\sigma|\mu)$ stands for the relative entropy of $\sigma$ with respect to $\mu$.

\begin{proposition}
\label{prop:vanishing-V-prim-transport}
Let $\alpha\in(1,2]$.
 The measure $\nu_{\alpha}$ satisfies the transport--entropy inequality $\mathbf{T}_{\min\{x^2, |x|^\alpha\}}(a)$ with some constant $a>0$ depending only on $\alpha$.
\end{proposition}

One can also wonder what happens if we allow the potential $V$ to have even bigger oscillations. For $\alpha>1$ and $\lambda >1$ define
\begin{equation*}
 V_{\alpha,\lambda}(x) = |x+ \lambda\sin(x)|^\alpha,\quad x\in\RR,
\end{equation*}
and let $\nu_{\alpha,\lambda}$ be the probability measure
 with density proportional to $\exp(-V_{\alpha,\lambda})$. 

\begin{proposition}
 \label{prop:vanishing-V-prim-big-lambda}
Let $\alpha>1$ and $\lambda >1$.  The measure $\nu_{\alpha,\lambda}$ does not satisfy 
the Poincar\'e inequality~\eqref{eq:Poinc-intro}.
\end{proposition}

\subsection{Proofs}

In the next two proofs  we shall omit the subscript $\alpha$ in the notation and write $V$, $\nu$, and $Z$ instead of $V_\alpha$, $\nu_{\alpha}$   and $Z_\alpha$, respectively.

\begin{proof}[Proof of Proposition~\ref{prop:vanishing-V-prim}.]
Fix  $\alpha >1$.
Let us start with proving that the Lata{\l}a--Oleszkiewicz inequality~\eqref{eq:LO} holds for $r\leq r_0(\alpha)$. For $x>0$ we have $V(x) = (x+\sin(x))^\alpha $ and
 \begin{equation*}
  V'(x) = \alpha (x+\sin(x))^{\alpha -1}(1+\cos(x)).
 \end{equation*}
 Denote for simplicity $\beta = (\alpha-1)/3$. For $x>0$ such that $1-\beta x^{-\beta-1}\geq 0$, we have
 \begin{align}
  \notag \int_x^\infty e^{-V(t)} dt
&=   \int_x^{x+x^{-\beta}} e^{-V(t)} dt +   \int_{x+x^{-\beta}}^\infty e^{-V(t)} dt\\
&\leq \frac{e^{-V(x)}}{x^{\beta}}  + \int_x^{\infty} e^{-V(u+u^{-\beta})} (1-\beta u^{-\beta-1} )du, \notag\\
&\leq \frac{e^{-V(x)}}{x^{\beta}}  + \int_x^{\infty} e^{-V(u+u^{-\beta})} du, \label{eq:exp--V}
 \end{align}
where we used the fact the $V$ is increasing on $(0,\infty)$ and substituted $t=u+u^{-\beta}$ in the second integral. Note that, by the convexity of the function $x\mapsto x^{\alpha}$, $x>0$,
\begin{align*}
 V(u+u^{-\beta}) - V(u)
 &= \bigl(u +  u^{-\beta} +\sin(u +  u^{-\beta})\bigr)^\alpha - (u+\sin(u))^\alpha\\
  &\geq \alpha(u+\sin(u))^{\alpha-1}\bigl(u^{-\beta} +\sin(u +  u^{-\beta}) -\sin(u)\bigr)\\
  &= \alpha(u+\sin(u))^{\alpha-1}\bigl(u^{-\beta} +2\sin(u^{-\beta}/2)\cos(u+u^{-\beta}/2)\bigr)\\
  &\geq \alpha(u+\sin(u))^{\alpha-1}\bigl(u^{-\beta} -2\sin(u^{-\beta}/2)\bigr)\\
  &\geq c_1 u^{\alpha-1} u^{-3\beta} = c_1
\end{align*}
for some $c_1=c_1(\alpha)>0$ (for, say, $u\geq1$). Thus~\eqref{eq:exp--V} implies that for sufficiently large $x$,
\begin{equation}
  \int_x^\infty e^{-V(t)} dt
\leq \frac{1}{1-e^{-c_1}}\cdot\frac{e^{-V(x)}}{x^{(\alpha-1)/3}}.
\label{eq:exp--V-final}
\end{equation}

We proceed similarly with $\int_0^x e^{V(t)} dt$. First note that, for $t\geq 2\pi$,
\[
 V(t-2\pi) + (2\pi)^\alpha = (t-2\pi + 
\sin(t))^{\alpha} + (2\pi)^{\alpha} \leq (t+\sin(t))^\alpha =  V(t),
\]
since $\alpha\geq 1$. Hence, 
\begin{equation*}
 \int_{2(k-1)\pi}^{2k\pi} e^{V(t)} dt = \int_{2k\pi}^{2(k+1)\pi} e^{V(t-2\pi)} dt \leq e^{-(2\pi)^\alpha} \int_{2k\pi}^{2(k+1)\pi} e^{V(t)} dt
\end{equation*}
and consequently, for $x\geq 2\pi$, 
 \begin{align*}
 \int_0^{x} e^{V(t)} dt &\leq \Bigl(1 +\sum_{k=0}^\infty e^{-k(2\pi)^\alpha} \Bigr)\int_{x-2\pi}^x e^{V(t)} dt\\
 &= \Bigl(1+\frac{1}{1-e^{-(2\pi)^\alpha}}\Bigr) \int_{x-2\pi}^x e^{V(t)} dt
\end{align*}
(recall that $V$ is increasing on $(0,\infty)$).

As above, denote for simplicity $\beta = (\alpha-1)/3$. For $x>2\pi$ we have
 \begin{align}
  \notag \int_{x-2\pi}^x e^{V(t)} dt
&=   \int_{x-x^{-\beta}}^x e^{V(t)} dt +   \int_{x-2\pi}^{x-x^{-\beta}} e^{V(t)} dt\\
&\leq \frac{e^{V(x)}}{x^{\beta}}  + \int_{l(x)}^{x} e^{V(u-u^{-\beta})} (1+\beta u^{-\beta-1} )du, \notag\\
&\leq \frac{e^{V(x)}}{x^{\beta}}  + (1+\beta(x-2\pi)^{-\beta-1})\int_{x-2\pi}^x e^{V(u-u^{-\beta})} du, \label{eq:exp-+V}
 \end{align}
where we used the fact the $V$ is increasing on $(0,\infty)$ and substituted $t=u-u^{-\beta}$ in the second integral ($l(x)>x-2\pi$ is the unique number such that $x-2\pi = l(x) -l(x)^{-\beta}$). Note that
\begin{align*}
 V(u-u^{-\beta}) & - V(u)
 = \bigl(u -  u^{-\beta} +\sin(u -  u^{-\beta})\bigr)^\alpha - (u+\sin(u))^\alpha\\
 \leq &-\alpha\bigl(u -  u^{-\beta} +\sin(u -  u^{-\beta})\bigr)^{\alpha-1}\bigl(u^{-\beta} -\sin(u -  u^{-\beta}) +\sin(u)\bigr)\\
  = &-\alpha\bigl(u -  u^{-\beta} +\sin(u -  u^{-\beta})\bigr)^{\alpha-1}\bigl(u^{-\beta} - 2\sin(u^{-\beta}/2)\cos(u-u^{-\beta}/2)\bigr)\\
  \leq &-\alpha\bigl(u -  u^{-\beta} +\sin(u -  u^{-\beta})\bigr)^{\alpha-1}\bigl(u^{-\beta} - 2\sin(u^{-\beta}/2)\bigr)\\
  \leq  &-c_2 u^{\alpha-1} u^{-3\beta} = -c_2
\end{align*}
for some $c_2=c_2(\alpha)>0$ and sufficiently large $u>0$. Thus~\eqref{eq:exp-+V} implies that for sufficiently large $x$,
\begin{align*}
  \int_{x-2\pi}^x e^{V(t)} dt
&\leq \frac{e^{V(x)}}{x^{\beta}}  + (1+\beta(x-2\pi)^{-\beta-1})e^{-c_2}\int_{x-2\pi}^x e^{V(u)} du\\
&\leq \frac{e^{V(x)}}{x^{\beta}}  + e^{-\widetilde{c}_2}\int_{x-2\pi}^x e^{V(u)} du,
\end{align*}
for some $\widetilde{c}_2 = \widetilde{c}_2(\alpha)>0$. Thus, for sufficiently large $x$,
\begin{equation}
  \int_{0}^x e^{V(t)} dt
\leq \frac{1+\frac{1}{1-e^{-(2\pi)^\alpha}}}{1-e^{-\widetilde{c}_2}}\cdot\frac{e^{V(x)}}{x^{(\alpha-1)/3}}.
\label{eq:exp-+V-final}
\end{equation}

For $q>2$ the function $t\mapsto t\log^{2/q}(1+1/(2t))$ is increasing for small enough positive $t$. Using \eqref{eq:exp--V-final} and \eqref{eq:exp-+V-final}, we see that for sufficiently large $x>0$,
\begin{align*}
 \MoveEqLeft
 \nu([x,\infty]) \log^{2/r'}\Bigl(1+\frac{1}{2\nu([x,\infty))} \Bigr) \int_0^x e^{V(t)} dt\\
 &\lesssim
 \frac{e^{-V(x)}}{x^{(\alpha-1)/3}} \log^{2/r'}\Bigl(Z e^{V(x)} x^{(\alpha-1)/3}\Bigr)  \frac{e^{V(x)}}{x^{(\alpha-1)/3}} 
 \lesssim 
  \frac{V(x)^{2/r'}}{x^{2(\alpha-1)/3}}  \lesssim x^{2(\alpha/r' - (\alpha-1)/3)}
\end{align*}
(we omit multiplicative constants not depending on $x$). Clearly, if $1<r\leq r_0(\alpha)$, then this is bounded as $x\to\infty$, and by the Barthe--Roberto criterion (see~\eqref{eq:LO-crit} above) the Lata{\l}a--Oleszkiewicz inequality with parameter $r$ does hold.

\medskip
Conversely, let us show that if $\nu$ enjoys  the Lata{\l}a--Oleszkiewicz with parameter $r$
then necessarily $r\le r_0(\alpha)$.  This can be seen by focusing on the points  $x_k=(2k+1)\pi$ where  $V'$ vanishes and our estimates can be reversed up to multiplicative constants. 
 Indeed, for $k$ large enough one can find a constant $c_3$ such that for $y\in[x_k-k^{-\beta}, x_k+k^{-\beta}]$,
 \begin{equation*}
  |V'(y)| \leq c_3 x_k^{\alpha-1} k^{-2\beta}.
\end{equation*}
For $k$ and $y$ as above,
\begin{align*}
 |V(y) - V(x_k)| \leq c_3 x_k^{\alpha-1} k^{-3\beta}\le c_4,
\end{align*}
using here that $\beta=(\alpha-1)/3$. Thus,
\begin{align*}
 \int_{x_k}^\infty e^{-V(x)} dx &\geq \int_{x_k}^{x_k+k^{-\beta}} e^{-V(x)}dx \geq k^{-\beta} e^{-V(x_k)}e^{-c_4},\\
 \int_0^{x_k} e^{V(x)} dx &\geq \int_{x_k-k^{-\beta}}^{x_k } e^{V(x)}dx \geq k^{-\beta} e^{V(x_k)}e^{-c_4}. 
\end{align*}
Consequently, as the function $t\mapsto t\log^{2/r'}(1+1/(2t))$ is increasing for small enough positive $t$, for $k$ sufficiently large we can write:
\begin{align*}
\MoveEqLeft
\nu([x_k, \infty)) \log^{2/r'}\Bigl(1+\frac{1}{2\nu([x_k, \infty))} \Bigr)\int_0^{x_k} e^{V(t)} dt\\
&\gtrsim
\frac{e^{-V(x_k)}}{k^\beta} \log^{2/r'}\Bigl( e^{V(x_k)} k^\beta\Bigr) \frac{e^{V(x_k)}}{k^\beta} \\
&\geq
\frac{ {V(x_k)}^{2/r'}}{k^{2\beta}}   \ge \frac{(2k\pi)^{2\alpha/r'}}{k^{2\beta}}.
\end{align*}
If $\nu$ satisfies the Lata{\l}a--Oleszkiewicz Inequality with parameter $r$, then by the 
 Barthe--Roberto criterion (see~\eqref{eq:LO-crit} above) the latter quantity remains bounded from above when $k\to\infty$. This forces $\alpha/r'\le \beta$, or equivalently $r\le r_0(\alpha)$.
 
\medskip
Next we turn to the proof of $(i)\Longleftrightarrow(iii)$. In view of Theorem \ref{thm:main:Toulouse}, we just need to prove that if $\nu$ satisfies a modified log-Sobolev inequality with parameter $r$ then necessarily $r\le r_0(\alpha)$. We apply the necessity part of
the criterion of Barthe--Roberto. It requires Assumption \eqref{eq:hyp-crit-mLS}, which is verified since \eqref{eq:exp-+V-final} is valid for $(r-1)V$ instead of $V$, with different numerical constants. Then we use the fact that the quantity $B^+_{mLS(r)}$ defined in \eqref{eq:BmLS} is bounded, together with lower bounds of $\int_{x_k}^\infty e^{-V}$
and $\int_0^{x_k} e^{(r-1)V}$. The computations are similar than the above ones for the Lata{\l}a--Oleszkiewicz inequality, and yield $r\le r_0(\alpha)$. We omit the details.
\end{proof}

\begin{proof}[Proof of Proposition~\ref{prop:vanishing-V-prim-transport}]
 Fix $\alpha\in(1,2]$ and denote, for $t\geq 0$,
 \[
  e^{-N(t)} = \frac{2}{Z} \int_{t}^\infty e^{-V(s)} ds.
 \]
Let $F_\nu$ and $F_{\exp}$ be the cumulative distribution functions of $\nu$ and the symmetric exponential measure with density $\frac{1}{2}e^{-|x|}$ respectively.

By Proposition~\ref{prop:vanishing-V-prim}, $\nu$ satisfies the Lata\l{}a--Oleszkiewicz inequality with $r_0(\alpha) >1$, so it satisfies the Poincar\'e inequality.  Thus, by Theorem~1.1 of \cite{GozW2}, in order to prove the assertion it suffices to show that  there exists $b = b(\alpha)>0$ such that 
 \[
  \bigl| F_{\nu}^{-1}(F_{\exp}(x)) -  F_{\nu}^{-1}(F_{\exp}(y)) \bigr|  \leq \frac{1}{b} (1+|x-y|)^{1/\alpha}
 \]
for all $x,y\in\RR$.

Note that for $x\geq 0$ we have $\tilde{x} = F_{\nu}^{-1}(F_{\exp}(x))$ if and only if
\[
 1 - \frac{1}{2} e^{-N(\tilde{x})}  = F_{\nu}(\tilde{x}) = F_{\exp}(x) = 1 - \frac{1}{2} e^{-x}.
\]
Thus it suffices to check whether there exists $b=b(\alpha)>0$ such that 
 \begin{equation}\label{eq:goal}
  b^\alpha \bigl| x - y\bigr|^\alpha  \leq 1+\bigl|N(|x|)\sgn(x) - N(|y|)\sgn(y)\bigr|
 \end{equation}
for all $x,y\in\RR$ (recall that $\nu$ is symmetric).

 If $|x-y|\leq 2\pi$, then one can guarantee that \eqref{eq:goal} holds simply by taking $b\leq (2\pi)^{-1}$. Let therefore consider the case when $|x-y|\geq 2\pi$. We have to cases:
\begin{enumerate}[1.]
\item $x,y$ are of different signs,
 \item $x,y$ are of the same sign.
\end{enumerate}

\emph{Case 1.} In the first case we have
\[
 b^\alpha |x-y|^\alpha = b^\alpha (|x|+|y|)^\alpha \leq 2^{\alpha-1} b^\alpha (|x|^\alpha + |y|^\alpha).
\]
From the proof of Proposition~\ref{prop:vanishing-V-prim} we know that for sufficiently large $t>0$ we have
\[
 \frac{1}{2} e^{-N(t)} = \frac{1}{Z} \int_t^{\infty} e^{-V(s)} ds \leq \frac{1}{Z(1-e^{-c_1(\alpha)})} \cdot\frac{e^{-V(t)}}{t^{(\alpha-1)/3}}
\]
(see \eqref{eq:exp--V-final}). Thus, for sufficiently large $t>0$,
\[
 N(t) \geq V(t) = |t+\sin(t)|^\alpha \geq \frac{1}{2} t^\alpha.
\]
Therefore, we can choose $b>0$ to be such that $2^{\alpha-1} b^\alpha t^\alpha \leq N(t) + \frac{1}{2}$ holds for all $t>0$. Then
\[
  b^\alpha |x-y|^\alpha \leq 2^{\alpha-1} b^\alpha (|x|^\alpha + |y|^\alpha) \leq 1+ N(|x|) + N(|y|),
\]
which is exactly \eqref{eq:goal} in the case when $x,y$ are of different signs.

\emph{Case 2.}  Suppose now that $x,y$ are of the same sign, say $x\geq y+2\pi \geq y \geq 0$.
Observe that, for $t>0$ and $k\in\NN$,
\[
 V(t+ 2k\pi)  = |t+ \sin(t) + 2k\pi|^\alpha \geq V(t) + (2k\pi)^\alpha.
\]
Thus, for $t>0$ and $s\geq 2\pi$,
\begin{equation*}
 V(t +s) \geq V(t + 2\pi\lfloor s /(2\pi)\rfloor) \geq V(t) + (2\pi)^\alpha \lfloor s/(2\pi)\rfloor^\alpha \geq  V(t) + s^\alpha/2^\alpha.
\end{equation*}
Therefore, for $t>0$ and $s\geq 2\pi$,
\begin{align*}
 e^{-N(t+s)} &= \frac{2}{Z} \int_{t+s}^{\infty} e^{-V(u)} du= \frac{2}{Z} \int_{t}^{\infty} e^{-V(u+s)} du\\
 &
 \leq \frac{2}{Z} \int_{t}^{\infty} e^{-V(u) - s^\alpha/2^\alpha} du = e^{-N(t) -  s^\alpha/2^\alpha}.\end{align*}
 Thus, if we take $b\leq 1/2$, then \eqref{eq:goal} holds also for $x\geq y+2\pi \geq y\geq 0$ (we substitute $x=t+s$, $y=t$).
 This finishes the proof.
\end{proof}

\begin{proof}[Proof of Proposition \ref{prop:vanishing-V-prim-big-lambda}]
Fix $\lambda,\alpha>1.$ We write $V$ for $V_{\alpha,\lambda}$. For sufficiently large $x>0$ we have
\[
 V'(x) =\alpha(x + \lambda \sin(x))^{\alpha-1}(1+\lambda\cos(x)). 
\]
Denote $x_k = (2k+1)\pi$, $k\in\NN$. There exists $\delta_0>0$ such that $V$ is decreasing on $[x_k-\delta_0, x_k+\delta_0]$ for sufficiently large $k$. Thus
\begin{equation}\label{eq:big-lambda-Muck-1}
 \int_{x_k}^\infty e^{-V(t)} dt \geq \int_{x_k}^{x_k + \delta_0} e^{-V(t)} dt \geq \delta_0 e^{-V(x_k)}.
\end{equation}

Moreover, by the convexity of the function $x\mapsto x^{\alpha}$ ($x>0$), for $h\in[0,\delta_0]$ we have
\begin{align*}
 V(x_k -h)  - V(x_k) &=  (x_k-h + \lambda \sin(x_k-h))^{\alpha} - (x_k + \lambda \sin(x_k))^{\alpha}\\
 &\geq \alpha(x_k + \lambda \sin(x_k))^{\alpha-1} (-h + \lambda \sin(x_k-h) )\\
&=\alpha x_k^{\alpha-1} ( \lambda \sin(h)  - h).
\end{align*}
Hence
\begin{align}
 \int_0^{x_k} e^{V(t)} dt &\geq  \int_{x_k-\delta_0}^{x_k} e^{V(t)} dt= \int_0^{\delta_0} e^{V(x_k-h)} dh\notag\\
 &\geq \int_0^{\delta_0} e^{V(x_k)  + \alpha x_k^{\alpha-1} ( \lambda \sin(h)  - h)} dh.
 \label{eq:big-lambda-Muck-2}
\end{align}
 Putting together \eqref{eq:big-lambda-Muck-1} and \eqref{eq:big-lambda-Muck-2},
we observe that 
\[ 
\int_{x_k}^\infty e^{-V(t)} dt \times \int_0^{x_k} e^{V(t)} dt \ge \delta_0 \int_0^{\delta_0} e^{  \alpha x_k^{\alpha-1} ( \lambda \sin(h)  - h)} dh
\]
tends to infinity when $k\to \infty$, since $\lambda \sin(h)  - h>0$ for $h$ positive and small enough. By the  Muckenhoupt criterion (see \eqref{eq:BP}), we may conclude that 
$\nu$ cannot satisfy  any Poincar\'e inequality.
\end{proof}

\subsection{Remarks on a general setting}
In this final section, we show informally how the method used in the calculations of the previous section can be extended to more general families of measures. The main issue is to derive estimates of the quantities $\int_{x_0}^x e^V$ and $\int_{x}^\infty e^{-V}$. For shortness we do not treat separately upper and lower bounds. 

The classical approach is based on writing $\int e^V=\int \frac{1}{V'} \times V'e^{V}$ and on an integration by parts. It  works if there exits $x_0$ and $\varepsilon>0$ such that for $x\ge x_0$, $V$ is of class $\mathcal C^2$, $V'>0$ and 
$\left|\frac{V"}{(V')^2} \right|\le 1-\varepsilon$. In this case, up to multiplicative constants 
which depend on $\varepsilon$, for $x\ge x_0$,
\[ \int_{x}^\infty e^{-V}\approx  \frac{e^{-V(x)}}{V'(x) }, \qquad \int_{x_0}^x e^V \approx  \frac{e^{V(x)}}{V'(x) } \cdot \]

This approach cannot work if $V'$ vanishes for arbitrarily large values, as it was the case for $V_\alpha$. The approach that we used for $\nu_{\alpha}$ can still be applied in such situations,
when the potential is, in some sense, essentially increasing.
The key parameter at point $x$ is a number $\theta(x)>0$ so that for some constants $C\ge c>0$ (independent of $x$),
\begin{align*}
   &\forall y\in [x-\theta(x),x+\theta(x)], \quad       |V(y)-V(x)|\le C, \\
   &V(x+\theta(x))\ge V(x)+c, \\
    &V(x-\theta(x))\le V(x)+c.    
\end{align*}
In words, $V$ is essentially constant on $[x-\theta(x),x+\theta(x)]$, but does increase between
the left endpoint and the center, and between the center and the right endpoint. 
For the upper bound, one also needs $V$ to grow at least linearly: $V(x+K)\ge V(x)+c$.
Under some additional assumptions (e.g., $\theta'$ is small enough compared to $c$), one gets 
\[ \int_{x}^\infty e^{-V}\approx  \theta(x)e^{-V(x)}, \qquad \int_{x_0}^x e^V \approx  \theta(x) e^{V(x)} \cdot \]
Note that when $V'(x)>0$, $1/V'(x)$ is heuristically the scale at which $V$ moves by 1, which 
makes a connection with the classical approach. Let us also mention that for measures having the latter properties, the Lata{\l}a--Oleszkiewicz and the modified log-Sobolev inequality with parameters $r$ will be true simultaneously. Indeed if $\theta$ is bounded from above, then
Condition \eqref{eq:hyp-crit-mLS} should be verified. Then the quantities $B^+_{LO(r)}$ and 
$B^+_{mLS(r)}$ are comparable since for $x$ large
\begin{align*}
& \left(\int_x^\infty e^{-V}\right) \log^{2/r'}\Bigl(1+\frac{1}{\int_x^\infty e^{-V}} \Bigr)\int_0^{x} e^{V(t)} dt \approx \theta(x)^2 \big( V(x)+\log \theta(x)\big)^{2/r'}, \\
& \left(\int_x^\infty e^{-V}\right) \log\Bigl(\frac{1}{\int_x^\infty e^{-V}} \Bigr)\left(\int_0^{x} e^{(r-1)V(t)}dt\right)^{1/(r-1)}  \approx \theta(x)^{r'} \big( V(x)+\log \theta(x)\big).
\end{align*}

Let us conclude with a simple observation about potentials which are nonincreasing 
on infinitely many intervals (variants involving essentially nonincreasing ones can be written).

 \begin{lemma}
  Let $\mu$ be a probability measure on the real line with density proportional to $\exp(-V(x))$ for some locally bounded $V\colon\RR\to\RR$. Suppose that there exists $\varepsilon>0$ and a sequence of positive real numbers $x_n\to\infty$, such that $V$ is nonincreasing on $(x_n-\varepsilon, x_n+\varepsilon)$. Then $\mu$ does not satisfy the Lata{\l}a--Oleszkiewicz inequality~\eqref{eq:LO} with any parameter $r\in(1,2)$.
 \end{lemma}
  \begin{proof}
  From the assumption about the monotonicity of $V$ on the intervals $(x_n-\varepsilon,x_n+\varepsilon)$, we get
  \[
    \int_{x_n}^\infty e^{-V(t)} dt \int_0^{x_n} e^{V(t)} dt \geq \varepsilon e^{-V(x_n)}\cdot \varepsilon e^{V(x_n)} =\varepsilon^2.
  \]
  Moreover, $\log^{2/r'} \Bigl(1+\frac{1}{2\mu([x,\infty))}\Bigr) \to +\infty$ for any $r\in(1,2)$. Thus, by the Barthe--Roberto criterion (see~\eqref{eq:LO-crit} above) the Lata{\l}a--Oleszkiewicz inequality~\eqref{eq:LO} cannot hold, with any parameter $r\in(1,2)$.
  \end{proof}

The above result should be compared to Proposition \ref{prop:vanishing-V-prim-big-lambda}.
Observe that measures satisfying the hypotheses of the Lemma may verify a Poincar\'e inequality. This is the case for the potential $V(x)=\lfloor |x| \rfloor$ involving the integer part. This potential is constant on every interval $[k,k+1)$, $\in \mathbb N$. Nevertheless $V(x)$ is a bounded additive perturbation of the potential $|x|$ of the symmetric exponential distribution, hence
the associated measure satisfies a Poincar\'e inequality.


\def\cprime{$'$}
\providecommand{\bysame}{\leavevmode\hbox to3em{\hrulefill}\thinspace}
\providecommand{\MR}{\relax\ifhmode\unskip\space\fi MR }
\providecommand{\MRhref}[2]{%
  \href{http://www.ams.org/mathscinet-getitem?mr=#1}{#2}
}
\providecommand{\href}[2]{#2}

\end{document}